\documentclass[12pt]{amsart}

\usepackage{amsmath,amsthm,amssymb,mathdots}
\usepackage{setspace,shuffle}
\usepackage{color}
\allowdisplaybreaks

\newtheorem{theorem}{Theorem}[section]
\newtheorem{proposition}[theorem]{Proposition}
\newtheorem{definition}[theorem]{Definition}
\newtheorem{lemma}[theorem]{Lemma}

\newtheorem{conjecture}[theorem]{Conjecture}

\theoremstyle{remark}
\newtheorem{remark}[theorem]{Remark}

\numberwithin{equation}{section}

\newcommand{\R}{\mathbb R}
\newcommand{\Z}{{\mathbb Z}}
\newcommand{\N}{{\mathbb N}}
\newcommand{\C}{{\mathbb C}}
\newcommand{\Q}{{\mathbb Q}}

\newcommand{\kk}{\boldsymbol{k}}

\newcommand{\veta}{\boldsymbol{\eta}}

\newcommand{\x}{{\sf x}}

\title[Finite and symmetric colored multiple zeta values]{Finite and symmetric colored multiple zeta values and multiple harmonic $q$-series at roots of unity}
\author{Koji Tasaka}
\address[Koji Tasaka]{Aichi Prefectural University}
\email{tasaka@ist.aichi-pu.ac.jp}

\date{}
\begin{document}
\maketitle

\begin{abstract}
The Kaneko-Zagier conjecture states that finite and symmetric multiple zeta values satisfy the same relations.
In the previous works with H.~Bachmann and Y.~Takeyama, we proved that the finite and symmetric multiple zeta values are obtained as an `algebraic' and `analytic' limit at $q\rightarrow 1$ of certain multiple harmonic $q$-sums, and studied their relations in order to give partial evidence of the Kaneko-Zagier conjecture.
In this paper, we start with multiple harmonic $q$-sums of level $N$, which are $q$-analogues of the truncated colored multiple zeta values. 
We introduce our finite and symmetric colored multiple zeta values as an algebraic and analytic limit of the multiple harmonic $q$-sums of level $N$ and discuss a higher level (or a cyclotomic) analogue of the Kaneko-Zagier conjecture. 
\end{abstract}

\section{Introduction}
For each tuple of positive integers $\kk=(k_1,\ldots,k_r)\in \Z_{>0}^r$, Kaneko and Zagier \cite{KZ} introduce the finite multiple zeta value $\zeta^{\mathcal{A}}(\kk)$ as an element in the $\Q$-algebra $\mathcal{A}=\big(\prod_{p} \mathbb{F}_p\big) \big/ \big(\bigoplus_{p} \mathbb{F}_p\big) $ with $p$ running over all prime.
They also define the symmetric multiple zeta value $\zeta^{\mathcal{S}}(\kk)$ as an element in the quotient $\Q$-algebra $\mathcal{Z}/\zeta(2)\mathcal{Z}$ of the algebra $\mathcal{Z}$ generated by all multiple zeta values
\[ \zeta(k_1,\ldots,k_r)=\sum_{m_1>\cdots>m_r>0}\frac{1}{m_1^{k_1}\cdots m_r^{k_r}} \quad (k_1\in \Z_{\ge2}, k_2,\ldots,k_r\in \Z_{\ge1}).\]
They established an exciting conjecture on these two objects, stating that finite multiple zeta values satisfy the same $\Q$-linear relations as symmetric multiple zeta values and vice versa.
This conjecture (called the \emph{Kaneko-Zagier conjecture} in this paper) is far from being solved at the present time, but remarkably, by many authors, several relations among finite and symmetric multiple zeta values are found in the same forms, which provide partial evidence for the Kaneko-Zagier conjecture. 
See \cite{Kaneko} and \cite{Z16} for references.
 
In this paper, we aim at a generalization of the Kaneko-Zagier conjecture, replacing the above multiple zeta values with the colored multiple zeta values of level $N$, which are defined for $\kk=(k_1,\ldots,k_r) \in \Z_{>0}^r$ and $\veta=(\eta_1,\ldots,\eta_r)\in \mu_N^r$ with $(k_1,\eta_1)\neq (1,1)$ by
\[L\binom{\veta}{\kk}=L\binom{\eta_1,\ldots,\eta_r}{k_1,\ldots,k_r}=\sum_{m_1>\cdots >m_r>0} \frac{\eta_1^{m_1}\cdots \eta_r^{m_r}}{m_1^{k_1}\cdots m_r^{k_r}},\] 
where $\mu_N$ is the set of $N$-th roots of unity.
A counterpart of the finite and symmetric multiple zeta values for the colored ones will be obtained as an `algebraic' and an `analytic' limit at $q=1$ of certain multiple harmonic $q$-sums, which can be viewed as a generalization of the previous results \cite[Theorems 1.1 and 1.2]{BTT18} (see also \cite{BTT19,Takeyama}).

Our finite and symmetric colored multiple zeta values contain the ones introduced by Singer and Zhao \cite{SZ15} and Jarossay \cite{J18} as special cases (Remark \ref{rem:connection_other_models}).
We will prove some standard relations for our finite and symmetric colored multiple zeta values, such as reversal relations, harmonic relations and linear shuffle relations (Propositions \ref{prop:reversal}, \ref{prop:harmonic}, \ref{prop:shuffle_general} and \ref{prop:shuffle_general2}).
In any case, relations we obtain are the same shape, so we may expect that finite and symmetric colored multiple zeta values satisfy the same linear relations over $\Q(\zeta_N)$, which can be viewed as a higher level (or a cyclotomic) analogue of the Kaneko-Zagier conjecture (see Section 6).

The organization of this paper is as follows.
In Section 2, we define multiple harmonic $q$-sums of level $N$ and give their asymptotic formulas at ``$q\rightarrow 1$".
In Section 3, taking the main terms of the asymptotic formulas, we define our symmetric colored multiple zeta values of level $N$ for each \emph{class} $\alpha\in \Z/N\Z$ as elements in the polynomial ring $\C[T]$.
The notion of a class $\alpha\in \Z/N\Z$ naturally appears in this context.
Using the regularization relation of colored multiple zeta values, we show their independence from $T$ (namely, our symmetric colored multiple zeta values lie in $\C$). 
This independence proves that the existence of an analytic limit at ``$q\rightarrow1$" of multiple harmonic $q$-sums of level $N$, whose limiting values turn out to be our symmetric colored multiple zeta values of level $N$.
In Section 4, as a conjectural counterpart of symmetric colored multiple zeta values, we define finite colored multiple zeta values of level $N$ for each class $\alpha\in \Z/N\Z$ and prove that they are obtained as an algebraic limit at ``$q=1$" of multiple harmonic $q$-sums of level $N$.
In Section 5, we give relations for finite and symmetric colored multiple zeta values.
In Section 6, we aim at providing evidence of a higher level analogue of the Kaneko-Zagier conjecture stating that our finite and symmetric colored multiple zeta values of level $N$ with a class $\alpha\in \Z/N\Z$ satisfy the same relations.

\

\noindent\textbf{Acknowledgments.} 
The author is grateful to Henrik Bachmann and Yoshihiro Takeyama for very valuable discussions. 
The author is also very grateful to Jianqiang Zhao for helpful comment on the linear shuffle relation.
This work was partially supported by JSPS KAKENHI Grant Numbers 18K13393.

\section{Multiple harmonic $q$-sums}

\subsection{Notations}

Throughout this paper, for a positive integer $N$ we denote by $\mu_N$ the set of $N$-th roots of unity. 
For a positive integer $m$, we denote by $\zeta_m$ a primitive $m$-th root of unity and by $\Q(\zeta_m)$ its cyclotomic field.

The complex conjugate of $a\in \C$ is denoted by $\overline{a}$.
We often use the relation $\overline{\eta}=\eta^{-1}$ for $\eta\in \mu_N$.

As usual, the Kronecker delta is denoted by $\delta_{a,b}$.

For $\kk =(k_1,\ldots,k_r) \in \Z_{>0}^r$ and $\veta=(\eta_1,\ldots,\eta_r) \in \mu_N^r$, we call a tuple $\binom{\veta}{\kk}\in \mu_N^r\times \Z_{>0}^r$ an \emph{index}, $k_1+\cdots+k_r$ the \emph{weight} and $r$ the \emph{depth}. 
An index $\binom{\veta}{\kk}$ is called \emph{admissible} if $\binom{\eta_1}{k_1}\neq\binom{1}{1}$.
We allow the empty index $\emptyset$ to be the unique index $\binom{\veta}{\kk}$ of weight 0 and depth $0$.
For any function $F$ on indices, set $F(\emptyset)=1$.

\subsection{Multiple harmonic $q$-sums at roots of unity}

The \emph{multiple harmonic $q$-sum of level $N$} is defined for a positive integer $m$ and an index $\binom{\veta}{\kk}=\binom{\eta_1,\ldots,\eta_r}{k_1,\ldots,k_r}\in \mu_N^r\times \Z_{>0}^r$ by
\begin{equation}\label{eq:def_mhq}
  \setlength\arraycolsep{1pt}z_m\left(\begin{array}{c}\veta\\\kk\end{array};q\right)=\sum_{m>m_1>\cdots>m_r>0}\frac{\eta_1^{m_1}\cdots \eta_r^{m_r}}{[m_1]_q^{k_1}\cdots [m_r]_q^{k_r}} ,
\end{equation}
where $[m]_q=1+q+\cdots+q^{m-1}=\frac{1-q^m}{1-q}$ is the $q$-integer.
We will be interested in the values at $q=\zeta_m$ primitive $m$-th roots of unity:
\[\setlength\arraycolsep{1pt}z_m\left(\begin{array}{c}\eta_1,\ldots,\eta_r\\ k_1,\ldots,k_r\end{array};\zeta_m\right)= \sum_{m> m_1>\cdots>m_r>0} \prod_{j=1}^r \eta_j^{m_j} \left( \frac{1-\zeta_m}{1-\zeta_m^{m_j} }\right)^{k_j},\]
which lie in $\Q(\zeta_m,\zeta_N) \left(\subset  \Q(\zeta_{mN}) \right)$.
We note that the above sums are not well-defined, if $\zeta_m$ is not primitive.

It should be mentioned that in the previous works \cite{BTT18,BTT19}, we develop the theory of values of multiple harmonic $q$-sums of the Bradley-Zhao model at primitive roots of unity.
Since our aim of this paper is to generalize these results for higher levels, a natural object would be the Bradley-Zhao model for multiple harmonic $q$-sums of level $N$, which is defined by
\begin{equation}\label{eq:bd_model} \sum_{m>m_1>\cdots>m_r>0}\frac{\eta_1^{m_1}q^{(k_1-1)m_1}\cdots \eta_r^{m_r}q^{(k_r-1)m_r}}{[m_1]_q^{k_1}\cdots [m_r]_q^{k_r}} .
\end{equation}
Actually, we can prove (but their detailed proofs are omitted) that our main results of this paper (Theorems \ref{thm:connection_smzv} and \ref{thm:connection_fmzv}) hold for this model as well.
Our model \eqref{eq:def_mhq} however simplifies some of the proofs, because of its shape of the harmonic relation (see Remark \ref{rem:bradley_zhao}).
A weak point of our model is that we do not know the domain of absolute convergence at $m\rightarrow \infty$ (we do not need this in this paper), while the Bradley-Zhao model for the case $N=1$ is widely studied on this subject (see e.g. \cite[\S12]{Z16}).

\subsection{Algebraic setup}
To describe relations of \eqref{eq:def_mhq}, it is convenient to use the algebraic setup given by Hoffman \cite{H97} (see also \cite{AK04,R02,SZ15}).
Let 
\[ \mathfrak{A}=\mathfrak{A}_N=\Q\langle e_0,e_\eta \mid \eta\in\mu_N \rangle\]
be the non-commutative polynomial algebra over $\Q$ and set
\[ \mathfrak{A}^1=\Q+\sum_{\eta\in \mu_N} \mathfrak{A}e_\eta,\quad \mathfrak{A}^0=\Q+\sum_{\eta\in \mu_N} e_0 \mathfrak{A} e_\eta+\sum_{\substack{\eta,\xi\in \mu_N\\\xi\neq 1}} e_\xi\mathfrak{A}e_\eta.\]

For $k\ge1$ and $\eta\in \mu_N$, write $e_{k,\eta}=e_0^{k-1}e_\eta $.
The subring $\mathfrak{A}^1$ is then freely generated by $e_{k,\eta} \ (k\ge1,\eta\in \mu_N)$.
We define the \emph{harmonic product} $\ast:\mathfrak{A}^1\times \mathfrak{A}^1\rightarrow \mathfrak{A}^1$ as a $\Q$-bilinear map given inductively by
\[ e_{k,\eta}w\ast e_{l,\xi}w' = e_{k,\eta}(w\ast e_{l,\xi}w' )+ e_{l,\xi}(e_{k,\eta}w\ast w') +e_{k+l,\eta\xi}(w\ast w')\]
for $k,l\ge1,\eta,\xi\in\mu_N$ and $w,w' \in \mathfrak{A}^1$, with the initial condition $1\ast w=w\ast 1 = w$.
Equipped with the harmonic product, the vector space $\mathfrak{A}^1$ forms a commutative $\Q$-algebra and $\mathfrak{A}^0$ is a $\Q$-subalgebra.
We also denote by $\mathfrak{A}^1_\ast$ and $\mathfrak{A}^0_\ast$ the commutative $\Q$-algebras equipped with the harmonic product.

The standard technique to prove the harmonic relation for the multiple harmonic $q$-sums is applied (see e.g. \cite{H97}). 
\begin{proposition}\label{prop:stuffle_zm}
The $\Q$-linear map $\mathfrak{z}_m : \mathfrak{A}^1_\ast \longrightarrow  \C$ defined by 
\[ \mathfrak{z}_m(e_{k_1,\eta_1}\cdots e_{k_r,\eta_r})= \setlength\arraycolsep{1pt}z_m\left(\begin{array}{c}\eta_1,\ldots,\eta_r\\ k_1,\ldots,k_r\end{array};\zeta_m\right) \quad \mbox{and} \quad \mathfrak{z}_m(1)=1\]
is an algebra homomorphism.
\end{proposition}

For example, the harmonic product $e_{k,\eta}\ast e_{l,\xi}=e_{k,\eta}e_{l,\xi}+e_{l,\xi}e_{k,\eta}+e_{k+l,\eta\xi}$ corresponds to the identities
\begin{equation*}\label{eq:har_zm}
\begin{aligned}
\setlength\arraycolsep{1pt}z_m\left(\begin{array}{c}\eta\\ k\end{array};q\right)
\setlength\arraycolsep{1pt}z_m\left(\begin{array}{c}\xi\\ l\end{array};q\right) & =\sum_{m>n,n'>0} \frac{\eta^m\xi^n}{[n]_{q}^k[n']_{q}^l} \\
&= \left(\sum_{m>n>n'>0} +\sum_{m>n>n'>0}+\sum_{m>n=n'>0} \right) \frac{\eta^m\xi^n}{[n]_{q}^k[n']_{q}^l} \\
&=\setlength\arraycolsep{1pt}z_m\left(\begin{array}{c}\eta,\xi\\ k,l\end{array};q\right)+\setlength\arraycolsep{1pt}z_m\left(\begin{array}{c}\xi,\eta\\ l,k\end{array};q\right)+\setlength\arraycolsep{1pt}z_m\left(\begin{array}{c}\eta\xi\\ k+l\end{array};q\right).
\end{aligned}
\end{equation*}

\subsection{The values at $q\rightarrow 1$ and colored multiple zeta values}

The limiting values of the multiple harmonic $q$-sums as $q\rightarrow 1$ coincide with the \emph{truncated colored multiple zeta values}, which are defined for each integer $m>0$ and  index $\binom{\veta}{\kk}=\binom{\eta_1,\ldots,\eta_r}{k_1,\ldots,k_r}$ by
\[L_m\binom{\veta}{\kk}=\sum_{m>m_1>\cdots >m_r>0} \frac{\eta_1^{m_1}\cdots \eta_r^{m_r}}{m_1^{k_1}\cdots m_r^{k_r}}.\]
Since $\displaystyle\lim_{q\rightarrow 1}[m]_q=m$, we have
\[ \lim_{q\rightarrow 1}\setlength\arraycolsep{1pt}z_m\left(\begin{array}{c}\veta\\\kk\end{array};q\right) = L_m\binom{\veta}{\kk}.\]
The limit 
\[ L\binom{\veta}{\kk}=\lim_{m\rightarrow \infty} L_m\binom{\veta}{\kk}\]
exists, if and only if $\binom{\veta}{\kk}$ is admissible (see e.g. \cite[Proposition 1.1]{AK04}).
These values are called the \emph{colored multiple zeta values of level $N$} \cite{Z16} (also multiple $L$-values in \cite{AK04} and multiple polylogarithm values at roots of unity in \cite{G98}).

We briefly recall the harmonic regularized multiple zeta values. 
First of all, the truncated colored multiple zeta values satisfy the harmonic relations.
Namely, the $\Q$-linear map $L_m : \mathfrak{A}^1_\ast \longrightarrow  \C$ defined by 
\[ L_m(e_{k_1,\eta_1}\cdots e_{k_r,\eta_r})= L_m\binom{\eta_1,\ldots,\eta_r}{ k_1,\ldots,k_r} \quad \mbox{and} \quad L_m(1)=1\]
is an algebra homomorphism.
Note that $\displaystyle\lim_{m\rightarrow \infty}L_m(e_{k_1,\eta_1}\cdots e_{k_r,\eta_r})=L\tbinom{\eta_1,\ldots,\eta_r}{ k_1,\ldots,k_r}$ holds if and only if $e_{k_1,\eta_1}\cdots e_{k_r,\eta_r}\in \mathfrak{A}^0$.
In a similar way as \cite[Theorems 3.1 and 4.1]{H97}, it can be shown that $\mathfrak{A}^1_\ast \cong \mathfrak{A}^0_\ast [e_{1,1}]$.
Namely, for any $w\in \mathfrak{A}^1_\ast$ there exist $w_0,w_1,\ldots,w_n \in \mathfrak{A}^0$ such that
\[ w=w_0+w_1\ast e_{1,1}+w_2\ast e_{1,1}^{\ast2}+\cdots+ w_n \ast e_{1,1}^{\ast n}\]
(explicit formulas for $w_0,w_1,\ldots,w_n$ can be obtained in much the same way as \cite[Corollary 5]{IKZ06}).
Applying $L_m$ to the above $w$ we get
\[ L_m(w)=L_m(w_0)+L_m(w_1)L_m(e_{1,1})+\cdots +L_m(w_n)L_m(e_{1,1})^n.\]
Since $L_m(e_{1,1})=\log m+\gamma +O(m^{-1}\log m) \ (m\rightarrow \infty)$, this shows that for any index $\binom{\veta}{\kk}$ there are the unique polynomials $L_\ast \left(\begin{smallmatrix} \veta \\ \kk \end{smallmatrix};T  \right) \in \C[T]$ such that
\begin{equation*}\label{eq:reg_colored} 
L_m \binom{\veta}{\kk}= \setlength\arraycolsep{1pt}L_\ast \left(\begin{array}{c} \veta \\ \kk \end{array};\log m + \gamma  \right) +O\left(\frac{\log^J m}{m}\right) \quad (m\rightarrow \infty),
\end{equation*}
with a positive integer $J$ depending on $\kk$, where $\gamma$ is the Euler constant.
By definition, we have $L_\ast \left(\begin{smallmatrix} \veta \\ \kk \end{smallmatrix};T  \right) = L\tbinom{\veta}{\kk}$, if $\tbinom{\veta}{\kk}$ is admissible, and $L_\ast \left(\begin{smallmatrix} 1 \\ 1 \end{smallmatrix};T  \right) = T$.

\subsection{Asymptotic formula}

In Theorem \ref{thm:connection_smzv} below, we give another limiting value of \eqref{eq:def_mhq} at $q=e^{2\pi i/m}$ as $m\rightarrow \infty$ (this is the `analytic' limit mentioned in the introduction).
In order to do this, in this subsection we show the following asymptotic formula.

\begin{theorem}\label{thm:asym_zm}
For any $\kk=(k_1,\ldots,k_r)\in \Z_{>0}^r$ and $\veta=(\eta_1,\ldots,\eta_r)\in\mu_N^r$ we have
\begin{align*}
\setlength\arraycolsep{1pt}z_m \left(\begin{array}{c}\veta\\\kk\end{array};e^{\frac{2\pi i}{m}} \right)  &=\sum_{j=0}^r (-1)^{k_1+\cdots+k_j} (\eta_{1}\cdots \eta_j)^m \\
&\times \setlength\arraycolsep{1pt}L_\ast \left(\begin{array}{c}\overline{\eta}_j,\ldots,\overline{\eta}_{1} \\k_j,\ldots,k_1 \end{array};\log \frac{m}{\pi} + \gamma +\frac{\pi i}{2}  \right) 
\setlength\arraycolsep{1pt}L_\ast \left(\begin{array}{c} \eta_{j+1},\ldots,\eta_r \\k_{j+1},\ldots,k_r \end{array};\log \frac{m}{\pi} + \gamma -\frac{\pi i}{2}  \right)  \\
&+O\left(\frac{\log^J m}{m}\right) \quad (m\rightarrow \infty)
\end{align*}
with a positive integer $J$ depending on $\kk$, where $\gamma$ is the Euler constant and $\overline{\eta}$ means the complex conjugate of $\eta\in \C$.
\end{theorem}
\begin{proof}
Let $\zeta_m=e^{\frac{2\pi i}{m}}$ and write $[n]=\frac{1-\zeta_m^n}{1-\zeta_m}$.
Decomposing the set $\{(m_1,\ldots,m_r)\in \mathbb{Z}^r\mid m>m_1>\cdots>m_r>0 \}$ into the disjoint union
\[ \bigsqcup_{j=0}^r \{(m_1,\ldots,m_r)\in \mathbb{Z}^r \mid m>m_1>\cdots>m_j> \frac{m}{2}\ge m_{j+1}>\cdots>m_r>0\}\]
and then changing the summation variables as $m_{a}=m-n_{j+1-a} \ (1\le a\le j)$ for each $j\in\{0,1,\ldots,r\}$, we have
\begin{align*}
&\setlength\arraycolsep{1pt}z_m \left(\begin{array}{c}\eta_1,\ldots,\eta_r\\ k_1,\ldots,k_r\end{array};\zeta_m \right)  = \sum_{j=0}^r  \sum_{m> m_1>\cdots>m_j> \frac{m}{2}} \prod_{a=1}^j \frac{\eta_a^{m_a}}{[m_a]^{k_a}}  \sum_{\frac{m}{2}\ge m_{j+1}>\cdots>m_r>0} \prod_{a=j+1}^r \frac{\eta_a^{m_a}}{[m_a]^{k_a}} \\
&=\sum_{j=0}^r (-\overline{\zeta}_m)^{k_1+\cdots+k_j} (\eta_{1}\cdots \eta_j)^m \sum_{\frac{m}{2}> n_{j}>\cdots>n_1>0} \prod_{a=1}^j \frac{\overline{\eta}_a^{n_a}}{\overline{[n_a]}^{k_a}}  \sum_{\frac{m}{2}\ge m_{j+1}>\cdots>m_r>0} \prod_{a=j+1}^r \frac{\eta_a^{m_a}}{[m_a]^{k_a}}.
\end{align*}
where for the last equality we have also used $(1-\overline{\zeta}_m)/(1-\zeta_m) =-\overline{\zeta}_m$ and $\overline{\eta}=\eta^{-1}$ for $\eta\in\C$ such that $|\eta|=1$.
For the second term in the last equation, we write
\[z_m^+\binom{\eta_1,\ldots,\eta_r}{k_1,\ldots,k_r}   = \sum_{\frac{m}{2}\ge m_{1}>\cdots>m_r>0} \prod_{a=1}^r \frac{\eta_a^{m_a}}{[m_a]^{k_a}}.\]
Then the first term can be written in the form
\begin{align*}
& \sum_{\frac{m}{2}> n_{j}>\cdots>n_1>0} \prod_{a=1}^j \frac{\overline{\eta}_a^{n_a}}{\overline{[n_a]}_{\zeta_m}^{k_a}}  \\
 &=\begin{cases} \displaystyle\overline{z_m^+\binom{\eta_j,\ldots,\eta_1}{k_j,\ldots,k_1}} & m: \ \mbox{odd} \\ 
 \mbox{} \\
 \displaystyle\overline{z_m^+\binom{\eta_j,\ldots,\eta_1}{k_j,\ldots,k_1}}   -\left( \frac{1-\overline{\zeta}_m}{2}\right)^{k_j}\overline{\eta}_j^{\frac{m}{2}} \overline{z_m^+\binom{\eta_{j-1},\ldots,\eta_1}{k_{j-1},\ldots,k_1}} & m: \ \mbox{even} \end{cases}.
\end{align*}
Since $1-\overline{\zeta}_m=O(1/m) \ (m\rightarrow \infty)$, the desired result follows from the next lemma.
\end{proof}

\begin{lemma}\label{lem:asym_zm+}
For $\kk\in \Z_{>0}^r$ and $\veta\in \mu_N^r$, we have
\[ z_m^+\binom{\veta}{\kk}  = \setlength\arraycolsep{1pt}L_\ast \left(\begin{array}{c} \veta \\\kk \end{array};\log \frac{m}{\pi} + \gamma -\frac{\pi i}{2}  \right) +O\left(\frac{\log^J m}{m}\right) \quad (m\rightarrow \infty)\]
with a positive integer $J$ depending on $\kk$.
\end{lemma}
\begin{proof}
For an index $\kk=(k_1,\ldots,k_r) \in \Z_{>0}^r$ and $\veta =(\eta_1,\ldots,\eta_r)  \in \mu_N^r$, we use
\[ \frac{1-e^{\frac{2\pi i}{m}}}{1-e^{\frac{2\pi i m_a}{m}}} e^{\frac{\pi i}{m}(m_a-1)} = \frac{e^{-\frac{\pi i}{m}}-e^{\frac{\pi i}{m}}}{e^{-\frac{\pi i m_a}{m}}-e^{\frac{\pi im_a}{m}}} = \frac{\sin \frac{\pi}{m}}{\sin \frac{\pi m_a}{m}},\]
to obtain
\begin{align*}\label{eq:H_N^+} 
\notag z_m^+\binom{\veta}{\kk} &= \sum_{\frac{m}{2}\ge  m_1>\cdots>m_r>0} \prod_{a=1}^r \eta_a^{m_a} \left( \frac{1-e^{\frac{2\pi i}{m}}}{1-e^{\frac{2\pi i m_a}{m}}}\right)^{k_a}\\
&=  \sum_{\frac{m}{2}\ge m_1>\cdots>m_r>0} \prod_{a=1}^r \eta_a^{m_a} e^{-\frac{\pi i}{m}(m_a-1)k_a}\left( \frac{\sin \frac{\pi}{m}}{\sin \frac{\pi m_a}{m}} \right)^{k_a}\\
&=\left( e^{\frac{\pi i}{m}} \frac{m}{\pi} \sin \frac{\pi}{m} \right)^{k_1+\cdots+k_r} \sum_{\frac{m}{2}\ge m_1>\cdots>m_r>0} \prod_{a=1}^r \eta_a^{m_a}  e^{ -\frac{\pi i }{m} m_a k_a}\frac{\pi^{k_a}}{m^{k_a}}   \left(\sin \frac{\pi m_a}{m}\right)^{-k_a}.
\end{align*}
Letting
\begin{equation}\label{eq:g} 
g_k(x) = e^{-ixk}x^k\left(\sin x\right)^{-k} \quad \mbox{and} \quad A_m^+\binom{\eta_1,\ldots,\eta_r}{k_1,\ldots,k_r} = \sum_{\frac{m}{2}\ge m_1>\cdots >m_r>0} \prod_{a=1}^r \eta^{m_a}_a\frac{ g_{k_a}\left(\frac{\pi m_a}{m} \right)}{m_a^{k_a}},
\end{equation}
we have 
\[  z_m^+\binom{\eta_1,\ldots,\eta_r}{k_1,\ldots,k_r} = \left( e^{\frac{\pi i}{m}} \frac{m}{\pi} \sin \frac{\pi}{m} \right)^{k_1+\cdots+k_r} A_m^+\binom{\eta_1,\ldots,\eta_r}{k_1,\ldots,k_r}.\]

We now prove 
\begin{equation}\label{eq:asym_conv} 
A_m^+\binom{\eta_1,\ldots,\eta_r}{k_1,\ldots,k_r} =L\binom{\eta_1,\ldots,\eta_r}{k_1,\ldots,k_r}+O\left(\frac{\log^J m}{m}\right)  \quad (m\rightarrow \infty) 
\end{equation}
for the case $k_1=1$ and $\eta_1\neq 1$. 
The case $k_1\ge2$ is obtained from the same argument with Lemma 2.7 of \cite{BTT18}, so is omitted.
Let $R=[\frac{m}{2}]$.
One has
\begin{align*}
&A_m^+\binom{\eta_1,\eta_2,\ldots,\eta_r}{1,k_2,\ldots,k_r} - L\binom{\eta_1,\eta_2,\ldots,\eta_r}{1,k_2,\ldots,k_r} \\
&=\sum_{R\ge m_1>\cdots >m_r>0}\frac{  \prod_{a=1}^r \eta_a^{m_a} g_{k_a}\left(\frac{\pi m_a}{m} \right)-1}{m_1m_2^{k_2}\cdots m_r^{k_r}} -\sum_{m_1>R}\frac{\eta_1^{m_1}}{m_1} \sum_{m_1>\cdots >m_r>0} \prod_{a=2}^r \frac{   \eta_a^{m_a} }{m_a^{k_a}} ,
\end{align*}
where $g_{k_1}$ is to be $g_1$ (we keep it for convenience).
We write $I_1$ and $I_2$ for the above first and second term, respectively. 
We use the standard method of Abel's summation.
Let $S(n)=\sum_{a=1}^n \eta_1^a$ with $S(0)=0$.
For $I_1$, one computes
\begin{align*}
I_1&= \sum_{R\ge m_1>\cdots >m_r>0}\frac{ (S(m_1)-S(m_1-1))\eta_2^{m_2}\cdots \eta_r^{m_r} \left(\prod_{a=1}^r  g_{k_a}\left(\frac{\pi m_a}{m} \right)-1\right)}{m_1m_2^{k_2}\cdots m_r^{k_r}} \\
&= \sum_{R\ge m_1>\cdots >m_r>0}\frac{ S(m_1)\eta_2^{m_2}\cdots \eta_r^{m_r} \left(\prod_{a=1}^r  g_{k_a}\left(\frac{\pi m_a}{m} \right)-1\right)}{m_1m_2^{k_2}\cdots m_r^{k_r}} \\
&- \sum_{R\ge m_1+1>\cdots >m_r>0}\frac{ S(m_1)\eta_2^{m_2}\cdots \eta_r^{m_r} \left(\prod_{a=1}^r  g_{k_a}\left(\frac{\pi (m_a+\delta_{a,1})}{m} \right)-1\right)}{(m_1+1)m_2^{k_2}\cdots m_r^{k_r}} \\
&= \left( \sum_{R> m_1>\cdots >m_r>0} + \sum_{R= m_1>\cdots >m_r>0}\right) \frac{ S(m_1)\eta_2^{m_2}\cdots \eta_r^{m_r} \left(\prod_{a=1}^r  g_{k_a}\left(\frac{\pi m_a}{m} \right)-1\right)}{m_1m_2^{k_2}\cdots m_r^{k_r}}  \\
&-\left( \sum_{R> m_1>m_2>\cdots >m_r>0} + \sum_{R> m_1=m_2>\cdots >m_r>0} \right) \frac{ S(m_1)\eta_2^{m_2}\cdots \eta_r^{m_r} \left(\prod_{a=1}^r  g_{k_a}\left(\frac{\pi (m_a+\delta_{a,1})}{m} \right)-1\right)}{(m_1+1)m_2^{k_2}\cdots m_r^{k_r}}\\
&=\sum_{R> m_1>m_2>\cdots >m_r>0} \frac{ S(m_1)\eta_2^{m_2}\cdots \eta_r^{m_r} }{m_2^{k_2}\cdots m_r^{k_r}} \left(\frac{ \prod_{a=1}^r  g_{k_a}\left(\frac{\pi m_a}{m} \right)-1 }{m_1}- \frac{ \prod_{a=1}^r  g_{k_a}\left(\frac{\pi (m_a+\delta_{a,1})}{m} \right)-1}{m_1+1}\right)\\
&+  \sum_{R= m_1>\cdots >m_r>0}\frac{ S(m_1)\eta_2^{m_2}\cdots \eta_r^{m_r} \left(\prod_{a=1}^r  g_{k_a}\left(\frac{\pi m_a}{m} \right)-1\right)}{m_1m_2^{k_2}\cdots m_r^{k_r}} \\
&- \sum_{R> m_1=m_2>\cdots >m_r>0}\frac{ S(m_1)\eta_2^{m_2}\cdots \eta_r^{m_r} \left(\prod_{a=1}^r  g_{k_a}\left(\frac{\pi (m_a+\delta_{a,1})}{m} \right)-1\right)}{(m_1+1)m_2^{k_2}\cdots m_r^{k_r}} .
\end{align*}
We write $I_{1,1},I_{1,2}$ and $I_{1,3}$ for the above first, second and third term, respectively. 
Since $g_{k}(x) = 1-ikx+o(x^2)$ $(x\rightarrow +0)$, there exists a positive constant $C$ depending on $k$ such that $|g_{k}(\pi m_a/m) -1 | \le Cm_a/m$ for all integers $m_a$ and $m$ satisfying $m/2\ge m_a>0$.
Since 
\begin{equation}\label{eq:prod_to_sum}
\prod_{a=1}^r g_{k_a}\left(\frac{\pi m_a}{m}\right) -1= \sum_{a=1}^r \left(g_{k_a}\left(\frac{\pi m_a}{m}\right)-1\right) \prod_{b=a+1}^{r} g_{k_b}\left(\frac{\pi m_b}{m}\right) ,
\end{equation}
it follows from $0<|g_{k}(x)| = \left|\frac{x}{\sin x}\right|^k \le \left(\frac{\pi}{2}\right)^k$ on the interval $(0,\frac{\pi}{2}]$ that for any $m_1,\ldots,m_r\le R$ we have 
\begin{align*}
\left| \prod_{a=1}^r g_{k_a}\left(\frac{\pi m_a}{m}\right) -1\right| \le \frac{C_1}{m} \sum_{a=1}^r m_a
\end{align*}
with $C_1>0$ depending on $\kk$. 
Since $S(m)$ is periodic, it is bounded by a positive constant $C_2$ for any $m$.
Setting $m_1=R$, we have
\begin{align*}
\left| I_{1,2} \right| &=\left| \frac{S(R)}{R} \sum_{R>m_2>\cdots >m_r>0}\frac{ \eta_2^{m_2}\cdots \eta_r^{m_r} \left(\prod_{a=1}^r  g_{k_a}\left(\frac{\pi m_a}{m} \right)-1\right)}{m_2^{k_2}\cdots m_r^{k_r}} \right|\\
&\le \frac{C_2}{R} \sum_{R>m_2>\cdots >m_r>0}\frac{\left| \prod_{a=1}^r  g_{k_a}\left(\frac{\pi m_a}{m} \right)-1\right|}{m_2^{k_2}\cdots m_r^{k_r}}\\
&\le \frac{C_1C_2}{mR} \sum_{a=1}^r \sum_{R>m_2>\cdots >m_r>0}\frac{  m_a}{m_2^{k_2}\cdots m_r^{k_r}}.
\end{align*}
Using the inequalities
\[ \frac{m_a }{m_1m_2^{k_2}\cdots m_r^{k_r}}\le  \frac{1}{m_2^{k_2}\cdots m_r^{k_r}} \quad (\mbox{for}\ r\ge a\ge 1, \, m_1>m_2>\cdots>m_r>0) \]
and 
\[ \sum_{m>m_1>0}\frac{1}{m_1} \le 2\log m,\]
we see that
\begin{align*}
\left| I_{1,2} \right| &\le \frac{rC_1C_2}{m}\sum_{R>m_2>\cdots >m_r>0}\frac{  1}{m_2^{k_2}\cdots m_r^{k_r}}\\
&\le \frac{rC_1C_2}{m} \left(\sum_{R>s>0} \frac{1}{s} \right)^{r-1} \le \frac{rC_1C_2}{m}(2\log R)^{r-1}.
\end{align*}
Thus, $\left| I_{1,2} \right|=O(m^{-1}(\log m)^{r-1})$.
For $I_{1,3}$ one computes
\begin{align*}
\left| I_{1,3}\right|&=\left| \sum_{R> m_1=m_2>\cdots >m_r>0}\frac{ S(m_1)\eta_2^{m_2}\cdots \eta_r^{m_r} \left(\prod_{a=1}^r  g_{k_a}\left(\frac{\pi (m_a+\delta_{a,1})}{m} \right)-1\right)}{(m_2+1)m_2^{k_2}\cdots m_r^{k_r}} \right|\\
&\le C_2 \sum_{R> m_1=m_2>\cdots >m_r>0}\frac{ \left|\prod_{a=1}^r  g_{k_a}\left(\frac{\pi (m_a+\delta_{a,1})}{m} \right)-1\right|}{m_2^{k_2+1}m_3^{k_3}\cdots m_r^{k_r}} \\
&\le \frac{C_2}{m} \sum_{R> m_2>\cdots >m_r>0}\frac{ 2m_2+1+m_3+\cdots+m_r}{m_2^{k_2+1}m_3^{k_3}\cdots m_r^{k_r}} \\
&\le \frac{C_2'}{m} \sum_{R> m_2>\cdots >m_r>0}\frac{ 1}{m_2^{k_2}m_3^{k_3}\cdots m_r^{k_r}}\quad (\mbox{with}\ C_2'>0), 
\end{align*}
and hence, $\left| I_{1,3}\right|=O(m^{-1}(\log m)^{r-1})$.
From \eqref{eq:prod_to_sum} the term $I_{1,1}$ can be reduced to
\begin{align*}
&\sum_{R> m_1>m_2>\cdots >m_r>0} \frac{ S(m_1)\eta_2^{m_2}\cdots \eta_r^{m_r} }{m_2^{k_2}\cdots m_r^{k_r}} \left(\frac{ g_1\left(\frac{\pi m_1}{m} \right) -1}{m_1} - \frac{ g_{1}\left(\frac{\pi (m_1+1)}{m} \right)-1}{m_1+1}\right)\prod_{a=2}^r  g_{k_a}\left(\frac{\pi m_a}{m} \right)\\
&+\sum_{R> m_1>m_2>\cdots >m_r>0} \frac{ S(m_1)\eta_2^{m_2}\cdots \eta_r^{m_r} }{m_2^{k_2}\cdots m_r^{k_r}} \left(\frac{ 1}{m_1} - \frac{1}{m_1+1}\right) \sum_{a=2}^r \left(g_{k_a}\left(\frac{\pi m_a}{m} \right)-1\right) \prod_{b=a+1}^r  g_{k_b}\left(\frac{\pi m_b}{m} \right).
\end{align*}
We write $I_{1,1,1}$ and $I_{1,1,2}$ for the above first and second term, respectively. 
For $I_{1,1,1}$, let $f(x)= x^{-1}-(\tan x)^{-1}$.
Then we have
\begin{align*}
|I_{1,1,1}| &\le \frac{C}{m} \sum_{R> m_1>m_2>\cdots >m_r>0} \frac{ 1}{m_2^{k_2}\cdots m_r^{k_r}} \left|f\left(\frac{\pi (m_1+1)}{m}\right)-f\left(\frac{\pi m_1}{m}\right) \right|\\
&= \frac{C}{m} \sum_{R> m_2>\cdots >m_r>0} \frac{ 1}{m_2^{k_2}\cdots m_r^{k_r}} \left|f\left(\frac{\pi R}{m}\right)-f\left(\frac{\pi (m_2+1)}{m}\right) \right|
\end{align*}
with $C>0$.
Since the function $f(x)$ is positive and increasing on the interval $[0,\frac{\pi}{2}]$ ($f(\pi/2)=\pi/2$ is maximal), $f\left(\frac{\pi R}{m}\right)-f\left(\frac{\pi (m_2+1)}{m}\right) $ is bounded by a positive constant for all $0<m_2<R$.
Thus, we get $\left| I_{1,1,1}\right|=O(m^{-1}(\log m)^{r-1})$.
For $I_{1,1,2}$, it can be shown that there exists a positive constant $C_1'$ such that
\begin{align*}
|I_{1,1,2}|& \le \frac{C_1'C_2}{m} \sum_{R> m_1>m_2>\cdots >m_r>0} \frac{ 1 }{m_1^2m_2^{k_2}\cdots m_r^{k_r}}  \sum_{a=2}^r m_a\\
&\le \frac{rC_1'C_2}{m} \sum_{R> m_1>m_2>\cdots >m_r>0} \frac{ 1 }{m_1m_2^{k_2}\cdots m_r^{k_r}} .
\end{align*}
So $\left| I_{1,1,2}\right|=O(m^{-1}(\log m)^{r})$.
Since $\left| I_{1}\right|\le \left| I_{1,1,1}\right|+\left| I_{1,1,2}\right|+\left| I_{1,2}\right|+\left| I_{1,3}\right|$ we obtain 
\[ \left| I_{1}\right|=O\left(\frac{(\log m)^r}{m}\right) \quad (m\rightarrow \infty).\]

For $I_2$, one has
\begin{align*}
\left|I_2\right|&\le \left|\sum_{m_1>R} S(m_1)\left( \frac{1}{m_1}  -\frac{1}{m_1+1}\right)\sum_{m_1>\cdots>m_r}\prod_{a=2}^r \frac{\eta_a^{m_a}}{m_a^{k_a}}\right| \\
&+\left|\frac{S(R)}{R+1}\sum_{R>m_2>\cdots>m_r}\prod_{a=2}^r \frac{\eta_a^{m_a}}{m_a^{k_a}} \right|\\
&\le  C \left( \sum_{m_1>\cdots >m_r>0 } - \sum_{R>m_1>\cdots>m_r>0} \right) \frac{1}{m_1^2m_2\cdots m_r} + C'\frac{(\log m)^{r-1}}{m}
\end{align*}
for some $C,C'>0$, which shows $\left|I_2\right|=O(m^{-1}(\log m)^r)$.
As a result, we obtain
\[\left| A_m^+\binom{\eta_1,\eta_2,\ldots,\eta_r}{1,k_2,\ldots,k_r} - L\binom{\eta_1,\eta_2,\ldots,\eta_r}{1,k_2,\ldots,k_r}\right| =O\left(\frac{(\log m)^r}{m}\right) \quad (m\rightarrow \infty),\]
so does \eqref{eq:asym_conv}.

By a similar argument with the definition of the harmonic regularized multiple zeta values in Section 2.4, the desired result is deduced from \eqref{eq:asym_conv}, Lemma 2.8 of \cite{BTT18}, that is,
\begin{equation}\label{eq:A_m^+1} 
A_m^+\binom{1}{1} =  \setlength\arraycolsep{1pt}L_\ast \left(\begin{array}{c} 1 \\ 1 \end{array};\log \frac{m}{\pi} + \gamma -\frac{\pi i}{2}  \right) +O\left(\frac{\log^J m}{m}\right)  \quad (m\rightarrow \infty),
\end{equation}
and the harmonic product formula for $z_m^+$ which is the same form with $L_m$.
\end{proof}
 
\begin{remark}\label{rem:bradley_zhao}
The proof of \eqref{eq:asym_conv} works for other models of multiple harmonic $q$-sums of level $N$.
For an index $\binom{\veta}{\kk}=\binom{\eta_1,\ldots,\eta_r}{k_1,\ldots,k_r}\in \mu_N^r\times \Z_{>0}^r$ and ${\bf a}=(a_1,\ldots,a_r)\in \R^r$, define 
\[\setlength\arraycolsep{1pt} z_m^+  \left(\begin{array}{c} \veta \\ \kk \end{array} ;{\bf a};q\right) = \sum_{\frac{m}{2}\ge m_1>\cdots>m_r>0} \prod_{j=1}^r \frac{\eta_j^{m_j}q^{a_jm_j}}{[m_j]^{k_j}}.\]
By a similar argument to the proof of Lemma \ref{lem:asym_zm+}, one can prove
\[\setlength\arraycolsep{1pt} z_m^+  \left(\begin{array}{c} \veta \\ \kk \end{array} ;{\bf a};e^{\frac{2\pi i}{m}}\right) =  L\binom{\veta}{\kk}+ O\left(\frac{\log^J m}{m}\right) \quad (m\rightarrow \infty) \]
for all admissible index $\tbinom{\veta}{\kk}$.
Namely, if $\tbinom{\veta}{\kk}$ is admissible, the index ${\bf a}$ does not contribute to the above asymptotic formula.
For non-admissible cases, their asymptotic formulas depend on choices of ${\bf a}$.
However, for the case ${\bf a}=(k_1-1,\ldots,k_r-1)$ (the Bradley-Zhao model \eqref{eq:bd_model}) one can show the same asymptotic formula 
\[ \setlength\arraycolsep{1pt}z_m^+\left(\begin{array}{c} \veta \\ \kk \end{array} ;{\bf a};e^{\frac{2\pi i}{m}}\right)=  \setlength\arraycolsep{1pt}L_\ast \left(\begin{array}{c} \veta \\\kk \end{array};\log \frac{m}{\pi} + \gamma -\frac{\pi i}{2}  \right) +O\left(\frac{\log^J m}{m}\right) \quad (m\rightarrow \infty)\]
with the case ${\bf a}=(0,\ldots,0)$.
The proof is done by first replacing $g_k(x)$ in \eqref{eq:g} with $g_k(x)=e^{ix(k-2)}x^k (\sin x)^{-k}$ and then doing the same computations as in the proof of Lemme \ref{lem:asym_zm+}.
Key ingredients in this case are that \eqref{eq:A_m^+1} does not change and that the harmonic product formula for $z_m^+$, which can be derived from the same argument to the proof of Proposition 2.9 of \cite{BTT18}, has the same shape with the one for $L_m$ modulo the terms $O\left(\frac{\log m}{m}\right)$.
\end{remark}

\section{Symmetric colored multiple zeta value}

\subsection{Definition}

Replacing the term $\log \frac{m}{\pi} + \gamma $ in Theorem \ref{thm:asym_zm} with a variable $T$, we are led to the following definition.

\begin{definition}\label{def:SCMZV}
Let $\alpha\in \Z/N\Z$.
For each index $\binom{\veta}{\kk}=\binom{\eta_1,\ldots,\eta_r}{k_1,\ldots,k_r}\in\mu_N^r \times \Z_{>0}^r$, we define the \emph{symmetric colored multiple zeta value of level $N$} $L^{\mathcal{S}}_\alpha \binom{\veta}{\kk} $ by
\begin{align*} 
&L^{\mathcal{S}}_\alpha \binom{\veta}{\kk}= L^{\mathcal{S}}_\alpha \setlength\arraycolsep{1pt}\left(\begin{array}{c}\veta\\\kk\end{array};T\right)\\
&=  \sum_{j=0}^r (-1)^{k_1+\cdots+k_j} (\eta_1\cdots \eta_j)^\alpha \setlength\arraycolsep{1pt} L_\ast \left(\begin{array}{c}\overline{\eta}_j,\ldots,\overline{\eta}_1\\k_j,\ldots,k_1\end{array};T+\frac{\pi i}{2}\right) \setlength\arraycolsep{1pt} L_\ast \left(\begin{array}{c}\eta_{j+1},\ldots,\eta_r\\k_{j+1},\ldots,k_r\end{array};T-\frac{\pi i}{2}\right).
\end{align*}
\end{definition}

As an example, the symmetric colored multiple zeta value of depth 1 is given by
\[ L^{\mathcal{S}}_\alpha\binom{\eta}{k} = \setlength\arraycolsep{1pt} L_\ast \left(\begin{array}{c}\eta\\k\end{array};T-\frac{\pi i}{2}\right)  + (-1)^k \eta^\alpha \setlength\arraycolsep{1pt} L_\ast \left(\begin{array}{c}\overline{\eta}\\k\end{array};T+\frac{\pi i}{2}\right) .\]
So, $L_\alpha^{\mathcal{S}}\tbinom{1}{1}= -\pi i$ and $L^{\mathcal{S}}_\alpha\tbinom{\eta}{k} =  L\tbinom{\eta}{k}+(-1)^k\eta^\alpha L\tbinom{\overline{\eta}}{k}$ if $(k,\eta)\neq(1,1)$.

Our symmetric colored multiple zeta values are apparently defined as elements in the polynomial ring $\C[T]$, but we will prove their independence from $T$ (Theorem \ref{thm:sym_MZV} below).
As a consequence, we see that the limit $\displaystyle\lim_{m\rightarrow\infty} z_{mN+\alpha} \left(\begin{smallmatrix} \veta \\ \kk \end{smallmatrix};e^{\frac{2\pi i}{mN+\alpha}}  \right)$ exists and its value coincides with $L^{\mathcal{S}}_\alpha \binom{\veta}{\kk} $ (Theorem \ref{thm:connection_smzv} below).

In what follows, we first review the regularization relation of colored multiple zeta values from Racinet \cite{R02} (see also \cite{AK04}), and then prove the independence from $T$.

\subsection{Iterated integral expression of the colored multiple zeta value}

The colored multiple zeta values can be written as iterated integrals.
We let
\[ \omega_{0}(t)=\frac{dt}{t}\quad \mbox{and} \quad \omega_{a}(t) = \frac{dt}{a^{-1}-t} \quad (a\in \C-\{0\}).\]
For $a_1,\ldots, a_k\in\C$ and a path $\gamma:[0,1]\rightarrow \C$, we consider the possibly divergent integrals
\begin{align*}
&I_{\gamma} (a_1,\ldots,a_k) =\int_{0}^{1} \gamma^\ast \omega_{a_1}(t_1) \cdots \int_{0}^{t_{k-2}} \gamma^\ast\omega_{a_{k-1}}(t_{k-1}) \int_{0 }^{t_{k-1}}\gamma^\ast\omega_{a_k}(t_k).
\end{align*}
For $(a,b)\in \C^2$, we denote by $[a,b]$ the path $t\rightarrow a+t(b-a)$.
By expanding $1/(a^{-1}-t)$ into the geometric series and performing the integral repeatedly, we have the following proposition (see \cite[Theorem 2.2]{G98}, \cite[Proposition 2.2.2]{R02} and \cite[Eq. (5)]{AK04}).

\begin{proposition}\label{prop:iterated_integral}
Let $k_1,\ldots,k_r$ be positive integers and $z_1,\ldots,z_r\in\C$  satisfying $0<|z_i|\le 1 \ (1\le i\le r)$.
For all $z$ in the open unit disc in $\C$, we have
\[ L\binom{zz_1,z_2,\ldots,z_r}{k_1,k_2,\ldots,k_r} = I_{[0,z]}(0^{k_1-1},z_1,0^{k_2-1},z_1z_2,\ldots, 0^{k_r-1},z_1\cdots z_r),\]
where $0^k$ means a sequence of zeros repeated $k$ times.
In particular, when $z_1=\eta_1,z_2=\eta_2/\eta_1 ,\ldots,z_r=\eta_r/\eta_{r-1}$ with $\eta_1,\ldots,\eta_r\in \mu_N$ and $\binom{\eta_1}{k_1}\neq \binom{1}{1}$, the limit $z\rightarrow 1$ exists on both sides and we have
\[ L\binom{\eta_1,\frac{\eta_2}{\eta_1},\ldots,\frac{\eta_r}{\eta_{r-1}}}{k_1,k_2,\ldots,k_r} = I_{[0,1]}(0^{k_1-1},\eta_1,0^{k_2-1},\eta_2,\ldots, 0^{k_r-1},\eta_r).\]
\end{proposition}

\

Since the corresponding index $\tbinom{\eta_1,\ldots,\eta_r}{k_1,\ldots,k_r}$ to a word $e_{k_1,\eta_1}\cdots e_{k_r,\eta_r}\in\mathfrak{A}^0$ is admissible, one can define the $\Q$-linear map 
\begin{align*} 
\mathcal{I} : \mathfrak{A}^0&\longrightarrow \C \\
e_{a_1}\cdots e_{a_k} &\longmapsto I_{[0,1]}(a_1,\ldots,a_k)
\end{align*}
and $\mathcal{I}(1)=1$.
The map $\mathcal{I}:\mathfrak{A}^0 \rightarrow\C$ is an algebra homomorphism (see \cite{AK04,Chen,H97}) with respect to the \emph{shuffle product} $\shuffle:\mathfrak{A}\times \mathfrak{A}\rightarrow \mathfrak{A}$ given inductively by
\[ e_{a}w\shuffle e_bw' = e_a(w\shuffle e_bw' )+ e_b(e_aw\shuffle w') \]
for $a,b\in\{0\}\cup \mu_N$ and $w,w' \in \mathfrak{A}$, with the initial condition $1 \shuffle w=w\shuffle 1 = w$.
Equipped with the shuffle product, the vector space $\mathfrak{A}$ forms a commutative $\Q$-algebra and $\mathfrak{A}^1$, $\mathfrak{A}^0$ are $\Q$-subalgebras.
We write $\mathfrak{A}_\shuffle,\mathfrak{A}^1_\shuffle,\mathfrak{A}^0_\shuffle$ for commutative $\Q$-algebras with the shuffle product.

Let ${\bf p}: \mathfrak{A}\rightarrow \mathfrak{A}$ and ${\bf q}: \mathfrak{A}\rightarrow \mathfrak{A}$ be the $\Q$-linear isomorphisms defined for integers $ r\ge0, k_1,\ldots,k_{r+1}\ge0$ and $\eta_1,\ldots,\eta_r\in \mu_N$ by
\begin{equation}\label{eq:p} 
\begin{aligned}
{\bf p}(e_0^{k_1}e_{\eta_1}\cdots e_0^{k_r}e_{\eta_r}e_0^{k_{r+1}}) =e_0^{k_1}e_{\eta_1}   e_0^{k_2}e_{\eta_1\eta_2}\cdots  e_0^{k_r}e_{\eta_1\cdots \eta_{r}} e_0^{k_{r+1}},\\
{\bf q}(e_0^{k_1}e_{\eta_1}\cdots e_0^{k_r}e_{\eta_r}e_0^{k_{r+1}}) =e_0^{k_1}e_{\eta_1}   e_0^{k_2}e_{\overline{\eta}_1\eta_2}\cdots  e_0^{k_r}e_{\overline{\eta}_{r-1} \eta_{r}} e_0^{k_{r+1}} 
\end{aligned}
\end{equation}
with ${\bf p}(1 )={\bf q}(1 )=1$.
We easily see that ${\bf q}\circ {\bf p}={\bf p}\circ {\bf q}={\rm id}$.
By Proposition \ref{prop:iterated_integral} the identity
\begin{equation}\label{eq:LI} 
\mathcal{L} =\mathcal{I}\circ {\bf p} \quad (\mbox{equivalently, } \ \mathcal{L} \circ{\bf q}=\mathcal{I})
\end{equation}
holds on $\mathfrak{A}^0$, where the $\Q$-linear map $\mathcal{L} : \mathfrak{A}^0 \longrightarrow  \C$ is defined by 
\[ \mathcal{L}(e_{k_1,\eta_1}\cdots e_{k_r,\eta_r})= L\binom{\eta_1,\ldots,\eta_r}{ k_1,\ldots,k_r}\]
and $\mathcal{L}(1)=1$.
We remark that the maps ${\bf p}$ and ${\bf q}$ preserve $\mathfrak{A}^0$.

\subsection{Extensions of the maps $\mathcal{L}$ and $\mathcal{I}$}

There are unique extensions of the maps $\mathcal{L}$ and $\mathcal{I}$, regularizing the divergent series and the divergent integrals.
These are defined to be algebra homomorphisms
\[ \hat{\mathcal{L}} (\cdot; T) : \mathfrak{A}^1_\ast \rightarrow \C[T]\quad \mbox{and}\quad \hat{\mathcal{I}} (\cdot; T) : \mathfrak{A}_\shuffle \rightarrow \C[T] \]
such that the maps $\hat{\mathcal{L}}$ and $\hat{\mathcal{I}} $ extend the algebra homomorphisms $\mathcal{L}:\mathfrak{A}^0_\ast \rightarrow \C$ and $\mathcal{I}:\mathfrak{A}^0_\shuffle \rightarrow \C$, respectively, and send $\hat{\mathcal{I}}(e_1)=\hat{\mathcal{L}} (e_{1,1})=T$ and $\hat{\mathcal{I}} (e_0)=0$.

We already explained how we construct the map $\hat{\mathcal{L}} (\cdot; T)$ in Section 2.4.
From this, it follows that 
\[ \hat{\mathcal{L}} (e_{k_1,\eta_1}\cdots e_{k_r,\eta_r}; T) = \setlength\arraycolsep{1pt}L_\ast \left(\begin{array}{c} \eta_1,\ldots,\eta_r \\ k_1,\ldots,k_r \end{array};T \right) \qquad (\forall e_{k_1,\eta_1}\cdots e_{k_r,\eta_r}\in \mathfrak{A}^1).\]

As for $\hat{\mathcal{I}}$, since $\mathfrak{A}_\shuffle\cong \mathfrak{A}^1_\shuffle [e_0]$ and $ \mathfrak{A}_\shuffle^1\cong \mathfrak{A}^0_\shuffle [e_1]$, any element $w\in \mathfrak{A}_\shuffle$ can be uniquely written in the form
\begin{equation*}
w= \sum_{i=0}^m \sum_{j=0}^n e_0^{\shuffle i}  \shuffle w_{i,j} \shuffle e_{1}^{\shuffle j} \quad ( m=\deg_{e_0} w,\ n=\deg_{e_1} w,  \ w_{i,j}\in \mathfrak{A}^0).
\end{equation*}
With this, the polynomial $\hat{\mathcal{I}}(w;T)$ is given by
\[ \hat{\mathcal{I}}(w;T) = \sum_{j=0}^n \mathcal{I}(w_{0,j}) T^j,\]
which is the unique polynomial satisfying for $e_{k_1,\eta_1}\cdots e_{k_r,\eta_r}\in \mathfrak{A}^1$
\[ \lim_{z\rightarrow 1-0}(1-z)^{-\delta} (I_{[0,z]}(0^{k_1-1},\eta_1,\ldots,0^{k_r-1},\eta_r) - \hat{\mathcal{I}}(e_{k_1,\eta_1}\cdots e_{k_r,\eta_r};-\log(1-z)) )=0\]
with some $\delta>0$.
Note that since $\hat{\mathcal{I}}(e_0;T)=0$, we have  $\hat{\mathcal{I}}(w;T) = \hat{\mathcal{I}}({\rm reg}_0(w);T)$, where we denote by
\[ {\rm reg}_0:\mathfrak{A}_\shuffle \cong \mathfrak{A}_\shuffle^1[e_0] \rightarrow \mathfrak{A}^1_\shuffle\]
the algebraic projection that sends $e_0$ to 0.

For example, it follows that
\begin{equation*}\label{eq:shuffle-reg-ex}
\hat{\mathcal{I}}(e_1;T)=T, \quad \hat{\mathcal{I}}(e_0^{n-1}e_1;T)=I_{[0,1]}(0^{n-1},1)= \zeta(n) \quad (n\ge2).
\end{equation*}
For the details, we refer to e.g. \cite[\S2.2]{AK04}, \cite[\S2]{R02} and \cite[\S2.9-10]{G98}.

\subsection{Dual setup}

In order to describe the regularization relation by Racinet \cite{R02}, we shall work on a dual space of the shuffle Hopf algebra $(\mathfrak{A},\shuffle, \Delta)$, where the coproduct $\Delta : \mathfrak{A} \rightarrow \mathfrak{A}\otimes \mathfrak{A}$ is defined as the deconcatenation given by $\Delta(e_{a_1}\cdots e_{a_r} ) = \sum_{j=0}^r e_{a_1}\cdots e_{a_j}\otimes e_{a_{j+1}}\cdots e_{a_r}$.
For structures of the shuffle Hopf algebra, we refer the reader to \cite[Appendix]{Z16}.
 
Let 
\[\mathfrak{A}^\vee=\C[T]\langle\langle \x_0,\x_{\eta}\mid \eta\in\mu_N\rangle\rangle\]
be the non-commutative power series algebra over the polynomial ring $\C[T]$, equipped with the concatenation product.
For simplicity of notation, we write $(e_{a_1}\cdots e_{a_r})^\vee = \x_{a_1}\cdots \x_{a_r}$ and $\varnothing^\vee=1$, and denote an element $S\in \mathfrak{A}^\vee$ by
\[ S = \sum_{w\in \{e_0,e_\eta\mid \eta\in \mu_N\}^\times} S_w w^\vee =S_{\varnothing}+S_{e_0}\x_0 +S_{e_1}\x_1+\cdots \quad (S_w\in \C[T]) \]
where $\{e_0,e_\eta\mid \eta\in \mu_N\}^\times$ is the set of words consisting of letters $e_0,e_\eta \  ( \eta\in \mu_N)$ with the empty word $\varnothing$.

Consider the pairing
\[ \langle,\rangle : \mathfrak{A}^\vee \otimes \mathfrak{A} \rightarrow \C[T]\]
defined for $(S,w)\in \mathfrak{A}^\vee \times \mathfrak{A} $ by $\langle S, w\rangle = S_w$.
Dualizing $\shuffle$ on $\mathfrak{A}$, we define the shuffle coproduct $\Delta_\shuffle : \mathfrak{A}^\vee \rightarrow \mathfrak{A}^\vee \widehat{\otimes} \mathfrak{A}^\vee$ by
\begin{equation}\label{eq:delta_shuffle}
\Delta_\shuffle (S) = \sum_{w_1,w_2\in\{e_0,e_\eta\mid \eta\in \mu_N\}^\times} \langle S, w_1\shuffle w_2\rangle w_1^\vee \otimes w_2^\vee.
\end{equation}
The shuffle Hopf algebra $(\mathfrak{A},\shuffle,\Delta)$ is then topologically dual to the completed Hopf algebra $(\mathfrak{A}^\vee,\cdot,\Delta_\shuffle,\sigma)$, where 
\[ \sigma: \mathfrak{A}^\vee \rightarrow \mathfrak{A}^\vee \]
is the antipode that is a continuous anti-automorpshism given by $\sigma(\x_a)=-\x_a$ for all $a\in\{0\}\cup \mu_N$.

An element $S\in \mathfrak{A}^\vee$ is called group-like if $\Delta_\shuffle(S)=S\widehat{\otimes} S$ and $S_\varnothing=1$.
We show the following standard properties on $\Delta_\shuffle$, which potentially can be found in the literature.

\begin{proposition}\label{prop:shuffle_coproduct}
i) The shuffle coproduct $\Delta_\shuffle$ is a continuous algebra homomorphism and satisfies $\Delta_\shuffle(\x_a)=1\otimes \x_a+\x_a\otimes 1$ for all $a\in \{0\}\cup  \mu_N$.\\
ii) $S$ is group-like if and only if $\langle S,w\shuffle w'\rangle = \langle S,w\rangle \langle S,w'\rangle $ holds for any $w,w'\in \mathfrak{A}$, i.e. the map $\mathfrak{A}_\shuffle \rightarrow \C[T]$ given by $w\mapsto \langle S,w\rangle$ is an algebra homomorphism. \\
iii) The set of group-like elements in $\mathfrak{A}^\vee$ forms a group with the concatenation product and with the inverse given by the antipode $\sigma$.
\end{proposition}
\begin{proof}
i) It suffices to show that $\Delta_\shuffle (\x_a g )= \Delta_\shuffle (\x_a)\Delta_\shuffle (g)$ for all word $g\in \{\x_0,\x_\eta \mid \eta\in \mu_N\}^\times$ and $a\in\{0,\eta\mid \eta\in \mu_N\}$.
Since $\shuffle$ is graded by degree, we easily see that $\Delta_\shuffle (\x_a)= \x_a\otimes 1 + 1\otimes \x_a$.
Noting
\[ \langle \x_{a}g,e_bw\rangle = \delta_{a,b}  \langle g,w\rangle \]
for $w\in \{e_0,e_\eta\mid \eta\in \mu_N\}^\times$, we compute
\begin{align*}
\Delta_\shuffle (\x_ag )&= \sum_{w_1,w_2\in \{e_0,e_\eta\mid \eta\in \mu_N\}^\times } \langle \x_ag, w_1  \shuffle   w_2\rangle w_1^\vee\otimes w_2^\vee\\
&= \langle \x_ag,\varnothing \rangle 1\otimes 1+  \sum_{e_b u_1\in \{e_0,e_\eta\mid \eta\in \mu_N\}^\times} \langle \x_ag,   e_b u_1\rangle  \x_b u_1^\vee \otimes 1\\
&+ \sum_{e_cu_2\in \{e_0,e_\eta\mid \eta\in \mu_N\}^\times } \langle \x_ag,   e_c u_2\rangle 1\otimes \x_c u_2^\vee  \\
&+  \sum_{e_bu_1,e_cu_2\in \{e_0,e_\eta\mid \eta\in \mu_N\}^\times } \langle \x_ag,   e_bu_1\shuffle e_cu_2\rangle \x_b u_1^\vee \otimes \x_c u_2^\vee  \\
&=  0+ \x_ag \otimes 1+1\otimes \x_ag \\
&+  \sum_{u_1,e_cu_2\in \{e_0,e_\eta\mid \eta\in \mu_N\}^\times } \langle g,   u_1\shuffle e_cu_2\rangle \x_a u_1^\vee \otimes \x_c u_2^\vee \\
&+ \sum_{e_bu_1,u_2\in \{e_0,e_\eta\mid \eta\in \mu_N\}^\times } \langle g,   e_bu_1\shuffle u_2\rangle \x_b u_1^\vee \otimes \x_a u_2^\vee  \\
&=  \sum_{u_1,w_2\in \{e_0,e_\eta\mid \eta\in \mu_N\}^\times } \langle g,   u_1\shuffle w_2\rangle \x_a u_1^\vee \otimes w_2^\vee \\
&+ \sum_{w_1,u_2\in \{e_0,e_\eta\mid \eta\in \mu_N\}^\times } \langle g,   w_1\shuffle u_2\rangle w_1^\vee \otimes \x_a u_2^\vee  \\
&=\sum_{w_1,w_2\in \{e_0,e_\eta\mid \eta\in \mu_N\}^\times } \langle g,   w_1\shuffle w_2\rangle (\x_a \otimes 1+1\otimes \x_a) w_1^\vee \otimes w_2^\vee\\
&=(\x_a \otimes 1+1\otimes \x_a)\sum_{w_1,w_2\in \{e_0,e_\eta\mid \eta\in \mu_N\}^\times } \langle g,   w_1\shuffle w_2\rangle  w_1^\vee \otimes w_2^\vee\\
&= \Delta_\shuffle(\x_a)\Delta_\shuffle(g),
\end{align*}
which proves i).

ii) This follows from \eqref{eq:delta_shuffle}.

iii) We only need to show invertibility of group-like elements.
For this, it follows from the definition of the antipode that for a group-like element $S\in \mathfrak{A}^\vee$ we have $\sigma(S)S=S\sigma(S)=1$.
Hence, $\sigma(S)$ is inverse to $S$.
We complete the proof.
\end{proof}

Let $I_0=\mathfrak{A}^\vee \x_0$ be the (right) ideal of $\mathfrak{A}^\vee$ generated by $\x_0$.
It follows that the set 
\[ \{ \x_0^{k_1-1}\x_{\eta_1}\cdots \x_0^{k_r-1}\x_{\eta_r} \mid r\ge0,k_1,\ldots,k_r\ge1,\eta_1,\ldots,\eta_r\in \mu_N\}\]
forms a linear basis of $\mathfrak{A}^\vee/I_0$.
Hence, $\mathfrak{A}^\vee/I_0$ is isomorphic to the non-commutative formal power series algebra
\[ \mathfrak{A}^{1,\vee} = \C[T]\langle\langle \x_{k,\eta}\mid  k\in\N,\eta\in\mu_N\rangle\rangle\]
as vector spaces. 
Let 
\[\pi_1: \mathfrak{A}^\vee \rightarrow  \mathfrak{A}^\vee/I_0 \overset{\sim}{\rightarrow} \mathfrak{A}^{1,\vee}\] 
be the surjective linear map that sends $\x_0^{k_1-1}\x_{\eta_1}\cdots \x_0^{k_r-1}\x_{\eta_r} $ to $\x_{k_1,\eta_1}\cdots \x_{k_r,\eta_r}$.
For $S\in \mathfrak{A}^\vee$, the image $ \pi_1 (S) \in \mathfrak{A}^{1,\vee}$ can be written as
\[ \pi_1(S) = \sum_{w\in \{e_{k,\eta}\mid k\ge1, \eta\in \mu_N\}^\times} S_w w^\vee,\]
where $\{e_{k,\eta}\mid k\ge1, \eta\in \mu_N\}^\times$ denotes the set of words consisting of 
$e_{k,\eta}$ ($k\ge1, \eta\in \mu_N$) and write $\x_{k_1,\eta_1}\cdots \x_{k_r,\eta_r}=(e_{k_1,\eta_1}\cdots e_{k_r,\eta_r})^\vee$.

Dualizing the isomorphisms ${\bf p}$ and ${\bf q}$ defined in \eqref{eq:p}, we get linear isomorphisms $\tilde{\bf p}:\mathfrak{A}^\vee\rightarrow \mathfrak{A}^\vee$ and $\tilde{\bf q}:\mathfrak{A}^\vee\rightarrow \mathfrak{A}^\vee$ as follows: for $S\in \mathfrak{A}^\vee$ they are given by
\[ \tilde{\bf p} (S) =\sum_{w\in\{e_0,e_\eta\mid \eta\in \mu_N\}^\times} S_{{\bf q}(w)} w^\vee \quad\mbox{and}\quad \tilde{\bf q} (S) =\sum_{w\in\{e_0,e_\eta\mid \eta\in \mu_N\}^\times} S_{{\bf p}(w)} w^\vee.\]
It follows that $\tilde{\bf q}\circ\tilde{\bf p}=\tilde{\bf p}\circ \tilde{\bf q}={\rm id}$ and that
\begin{equation*} 
\begin{aligned}
\tilde{\bf p}(\x_0^{k_1}\x_{\eta_1}\cdots \x_0^{k_r}\x_{\eta_r}\x_0^{k_{r+1}}) =\x_0^{k_1}\x_{\eta_1}   \x_0^{k_2}\x_{\eta_1\eta_2}\cdots  \x_0^{k_r}\x_{\eta_1\cdots \eta_{r}} \x_0^{k_{r+1}},\\
\tilde{\bf q}(\x_0^{k_1}\x_{\eta_1}\cdots \x_0^{k_r}\x_{\eta_r}\x_0^{k_{r+1}}) =\x_0^{k_1}\x_{\eta_1}   \x_0^{k_2}\x_{\overline{\eta}_1\eta_2}\cdots  \x_0^{k_r}\x_{\overline{\eta}_{r-1} \eta_{r}} \x_0^{k_{r+1}} .
\end{aligned}
\end{equation*}
These induce linear isomorphisms on $\mathfrak{A}^{1,\vee}$.
Since they are commute with $\pi_1$, we also denote them by $\tilde{\bf p}:\mathfrak{A}^{1,\vee}\rightarrow \mathfrak{A}^{1,\vee}$ and $\tilde{\bf q}:\mathfrak{A}^{1,\vee}\rightarrow \mathfrak{A}^{1,\vee}$.
For example, we have
\[\tilde{\bf q} (\x_{k_1,\eta_1}\cdots \x_{k_r,\eta_r}) =\x_{k_1,\eta_1}   \x_{k_2,\overline{\eta}_1\eta_2}\cdots  
\x_{k_r,\overline{\eta}_{r-1} \eta_{r}} .\]

\subsection{Regularization relation}
Following \cite{R02}, we recall the regularization relation, which describes a difference between $\hat{\mathcal{L}}$ and $\hat{\mathcal{I}}$.

Consider
\[ \Phi_\shuffle (T) =  \sum_{w\in \{e_0,e_\eta\mid \eta\in \mu_N\}^\times} \hat{\mathcal{I}}(w;T) w^\vee \in \mathfrak{A}^\vee ,\]
and
\[ \Psi_\ast (T) =  \sum_{w\in \{e_{k,\eta}\mid k\ge1, \eta\in \mu_N\}^\times} \hat{\mathcal{L}}(w;T) w^\vee \in \mathfrak{A}^{1,\vee} ,\]
where we set $\hat{\mathcal{I}}(\varnothing;T)=\hat{\mathcal{L}}(\varnothing;T)=1$.
Since $\hat{\mathcal{I}}$ is an algebra homomorphism, i.e. $\langle \Phi_\shuffle(T),w\shuffle w'\rangle = \langle \Phi_\shuffle(T),w\rangle \langle \Phi_\shuffle(T),w'\rangle $ holds for any $w,w'\in \mathfrak{A}$, by Proposition \ref{prop:shuffle_coproduct} we see that $\Phi_\shuffle(T)$ is group-like.
A priori $\Psi_\ast(T)$ is not group-like and is different from $\pi_1\big(\Phi_\shuffle(T)\big)$ on $\mathfrak{A}^{1,\vee}$.
Their difference is described in \eqref{eq:reg_cmzv} below.

\begin{lemma}\label{lem:reg_theorem}
i) We have
\[ \Psi_\ast(T) = e^{T\x_{1,1}}\Psi_\ast(0) \quad \mbox{and}\quad \Phi_\shuffle(T) = e^{T\x_1}\Phi_\shuffle(0),\]
where $e^{T\x_{a}}=\exp(T\x_{a})=\sum_{n\ge0} \frac{\x_{a}^n}{n!}T^n$.\\
ii)
Let $\Lambda(x) = \exp\left( \sum_{n\ge2} \frac{(-1)^{n-1}}{n}\zeta(n)x^n\right)\in \R[[x]]$.
Then we have
\begin{equation}\label{eq:reg_cmzv}
\Psi_\ast (T) = \Lambda(\x_{1,1}) \tilde{\bf q} \big( \pi_1 (\Phi_\shuffle(T))\big).
\end{equation}
\end{lemma}
\begin{proof}
The statement i) follows from \cite[Corollaries 2.4.4 and 2.4.5]{R02} (see also \cite[Lemma 4.4]{SZ15}).\\
Now prove ii).
If $w$ is admissible (i.e. $w\in \mathfrak{A}^0$), the equality of the coefficient of $w$ in \eqref{eq:reg_cmzv} is equivalent to \eqref{eq:LI}.
In fact, one has
\[ \langle \Psi_\ast(T) ,w \rangle =\mathcal{L}(w) = \mathcal{I}({\bf p}(w))=\left\langle \tilde{\bf q}\big(\Phi_\shuffle(T)\big) ,w \right\rangle= \left\langle \Lambda(\x_{1,1}) \tilde{\bf q}\big(\pi_1(\Phi_\shuffle(T))\big) ,w \right\rangle.\]
A crucial difference between the shuffle and harmonic regularizations shows up if $w\in \mathfrak{A}^1 \backslash \mathfrak{A}^0$.
This is described as the regularization relation \cite[Corollary 2.4.15]{R02}.
In our setting, it is
\[ \Psi_\ast(0)= \Lambda(\x_{1,1})\tilde{\bf q}\big(\pi(\Phi_\shuffle(0))\big). \]
Since $\tilde{\bf q}(\x_{1,1}^nw)=\x_{1,1}^n\tilde{\bf q}(w)$, from i) one has
\begin{align*}
\Psi_\ast(T)&= e^{T\x_{1,1}}\Psi_\ast(0) =   e^{T\x_{1,1}}\Lambda(\x_{1,1})\tilde{\bf q}\big(\pi_1(e^{-T\x_1} \Phi_\shuffle(T))\big)\\
&=e^{T\x_{1,1}}\Lambda(\x_{1,1})e^{-T\x_{1,1}}  \tilde{\bf q}\big(\pi_1 (\Phi_\shuffle(T))\big)=\Lambda(\x_{1,1})\tilde{\bf q}\big(\pi_1(\Phi_\shuffle(T))\big),
\end{align*}
from which the statement ii) follows (see also \cite[Theorem 2.2]{AK04}).
\end{proof}

\subsection{Non-commutative generating series}

We now compute the non-commutative generating series of symmetric colored multiple zeta values.

Let us define
\[ \Phi_\ast(T) = \Lambda(\x_1)\Phi_\shuffle(T) \in \mathfrak{A}^\vee \]
and for $\alpha\in \Z/N\Z$ set
\[ \Xi_\alpha (T_1,T_2) =\sum_{\eta\in \mu_N} \overline{\eta}^\alpha \sigma( \Phi_\ast^\eta(T_1)) \x_\eta \Phi_\ast^\eta(T_2) ,\]
which lies in $ \C[T_1,T_2]\langle\langle \x_0,\x_{\eta}\mid \eta\in\mu_N\rangle\rangle$.
Here for $S\in \mathfrak{A}^\vee$ we write
\[ S^\eta = \sum_{w}S_{{\bf t}_\eta (w)} w^\vee \]
with ${\bf t}_\eta : \mathfrak{A}\rightarrow \mathfrak{A}$ being an algebra homomorphism with respect to the concatenation such that $ {\bf t}_\eta (e_{a}) =e_{a\overline{\eta}}$.
A similar generating series to $\Xi_{1} (0,0)$ was  introduced by Jarossay \cite[Appendix A]{J18} in a connection with $p$-adic symmetric colored multiple zeta values.

We remark that Lemma \ref{lem:reg_theorem} ii) shows the identities
\[ \tilde{\bf q}\big(\pi_1(\Phi_\ast(T))\big) = \Lambda(\x_{1,1})\tilde{\bf q}\big( \pi_1(\Phi_\shuffle(T))\big) = \Psi_\ast(T).\]
Hence, for $w\in \mathfrak{A}^1$ we have
\begin{equation}\label{eq:cof_Phi_ast} 
\langle \Phi_\ast(T), w\rangle = \langle \tilde{\bf p}\big(\Psi_\ast(T)\big), w\rangle = \hat{\mathcal{L}}({\bf q}(w);T).
\end{equation}

\begin{lemma}\label{lem:xi_alpha_generating_series}
For integers $k_1,\ldots,k_r\ge1$ and $\eta_1,\ldots,\eta_r\in \mu_N$, we have 
\begin{align*}
& \left\langle \Xi_\alpha (T_1,T_2), {\bf p} \big(e_1 e_0^{k_1-1}e_{\eta_1} e_0^{k_2-1}e_{\eta_2} \cdots e_0^{k_r-1}e_{\eta_r}\big) \right\rangle \\
&=\sum_{j=0}^r (-1)^{k_1+\cdots+k_j} (\eta_1\cdots \eta_j)^\alpha \setlength\arraycolsep{1pt} L_\ast \left(\begin{array}{c}\overline{\eta}_j,\ldots,\overline{\eta}_1\\k_j,\ldots,k_1\end{array};T_1\right) \setlength\arraycolsep{1pt} L_\ast \left(\begin{array}{c}\eta_{j+1},\ldots,\eta_r\\k_{j+1},\ldots,k_r\end{array};T_2\right).
\end{align*}
\end{lemma}
\begin{proof}
Let $w= e_{\eta_0}e_0^{k_1-1}e_{\eta_1} \cdots e_0^{k_r-1}e_{\eta_r}$.
Since 
\[\sum_{\eta\in \mu_N} \langle S\x_{\eta} S', w\rangle = \sum_{j=0}^r \langle S, e_{\eta_0}e_0^{k_1-1}e_{\eta_1} \cdots e_{\eta_{j-1}}e_0^{k_j-1}\rangle\langle S', e_0^{k_{j+1}-1}e_{\eta_{j+1}}\cdots e_0^{k_r-1}e_{\eta_r}\rangle\] 
holds for $S,S'\in \mathfrak{A}^\vee$, by \eqref{eq:cof_Phi_ast} one has
\begin{align*}
&\left\langle \Xi_\alpha (T_1,T_2), w\right\rangle \\
&=\sum_{j=0}^r  (-1)^{k_1+\cdots+k_j}\overline{\eta}_j^\alpha \hat{\mathcal{L}}({\bf q}\circ {\bf t}_{\eta_j}(e_0^{k_j-1}e_{\eta_{j-1}} \cdots e_0^{k_1-1}e_{\eta_0});T_1) \\
&\times\hat{\mathcal{L}} ({\bf q}\circ {\bf t}_{\eta_j} (e_0^{k_{j+1}-1}e_{\eta_{j+1}}\cdots e_0^{k_r-1}e_{\eta_r});T_2)\\
&= \sum_{j=0}^r  (-1)^{k_1+\cdots+k_j}\overline{\eta}_j^\alpha \hat{\mathcal{L}}({\bf q}(e_0^{k_j-1}e_{\eta_{j-1}\overline{\eta}_j} \cdots e_0^{k_1-1}e_{\eta_0\overline{\eta}_j});T_1)\\
&\times \hat{\mathcal{L}}( {\bf q}(e_0^{k_{j+1}-1}e_{\eta_{j+1}\overline{\eta}_j}\cdots e_0^{k_r-1}e_{\eta_r\overline{\eta}_j});T_2)\\
&= \sum_{j=0}^r  (-1)^{k_1+\cdots+k_j}\overline{\eta}_j^\alpha \hat{\mathcal{L}}(e_0^{k_j-1}e_{\eta_{j-1}\overline{\eta}_{j}} e_0^{k_{j-1}-1}e_{\eta_{j-2}\overline{\eta}_{j-1}}  \cdots e_0^{k_1-1}e_{\eta_0\overline{\eta}_1};T_1)\\
&\times \hat{\mathcal{L}}( e_0^{k_{j+1}-1}e_{\eta_{j+1}\overline{\eta}_{j}}e_0^{k_{j+2}-1}e_{\eta_{j+2}\overline{\eta}_{j+1}}\cdots e_0^{k_r-1}e_{\eta_{r}\overline{\eta}_{r-1}};T_2)\\
&= \sum_{j=0}^r (-1)^{k_1+\cdots+k_j}  \overline{\eta}_j^\alpha \setlength\arraycolsep{1pt} L_\ast \left(\begin{array}{c}\frac{\overline{\eta}_{j}}{\overline{\eta}_{j-1}} ,\ldots,\frac{\overline{\eta}_1}{\overline{\eta}_0} \\k_j,\ldots,k_1\end{array};T_1\right) \setlength\arraycolsep{1pt} L_\ast \left(\begin{array}{c}\frac{\eta_{j+1}}{\eta_{j}},\ldots,\frac{\eta_{r}}{\eta_{r-1}} \\k_{j+1},\ldots,k_r\end{array};T_2\right).
\end{align*}
Letting $\eta_0=1$ and replacing $(\eta_1,\eta_2,\ldots,\eta_r)$ with $\left(\eta_1,\eta_1\eta_2,\ldots,\eta_1\cdots \eta_r\right)$, we get the desired result.
\end{proof}

\begin{theorem}\label{thm:sym_MZV}
For any $\alpha\in \Z/N\Z$ and index $\binom{\veta}{\kk}\in \mu_N^r\times \Z_{>0}^r$, we have $L_{\alpha}^{\mathcal{S}}\tbinom{\veta}{\kk}\in \C$.
\end{theorem}
\begin{proof}
Note that the map ${\bf t}_\eta$ is an automorphism for any $\eta\in \mu_N$ and its inverse is ${\bf t}_{\overline{\eta}}$.
By definition, $\langle S^\eta ,w\rangle = \langle S ,{\bf t}_\eta(w) \rangle$ holds for all $S\in \mathfrak{A}^\vee$ and $w\in \mathfrak{A}$, and $\langle S_1S_2,w\rangle=\sum_{w=w_1w_2} \langle S_1,w_1\rangle \langle S_2,{w_2}\rangle$ holds for $S_1,S_2\in \mathfrak{A}^\vee$ and $w\in \mathfrak{A}$. 
With this, for $\eta\in \mu_N$, one has
\begin{align*}
\langle (S_1S_2)^\eta,w\rangle &=\langle S_1S_2, {\bf t}_{{\eta}} (w)\rangle  =\sum_{{\bf t}_{{\eta}} (w)=w_1w_2} \langle S_1,w_1\rangle \langle S_2,{w_2}\rangle\\
&=\sum_{w={\bf t}_{\overline{\eta}} (w_1){\bf t}_{\overline{\eta}}(w_2)}  \langle S_1,w_1\rangle \langle S_2,w_2\rangle =\sum_{w=v_1v_2 }  \langle S_1,{\bf t}_{{\eta}} (v_1) \rangle \langle S_2,{\bf t}_{{\eta}} (v_2) \rangle \\
&=\sum_{w=v_1v_2 }  \langle S_1^\eta,v_1 \rangle \langle S_2^\eta,v_2 \rangle= \langle S_1^\eta S_2^\eta,w\rangle,
\end{align*}
so $ (S_1S_2)^\eta=S_1^\eta S_2^\eta $.
This shows
\[\Phi_\ast^{\eta}(T) = \Lambda(\x_1)^\eta \Phi_\shuffle^\eta(T) =\Lambda(\x_\eta ) \Phi_\shuffle^\eta(T).\]
Using Lemma \ref{lem:reg_theorem} i), one computes
\begin{align*}
\Xi_\alpha  (T_1,T_2) &= \sum_{\eta\in \mu_N} \overline{\eta}^\alpha \sigma\big(  \Lambda(\x_\eta)\Phi_\shuffle^\eta(T_1) \big) \x_\eta  \Lambda (\x_\eta)\Phi_\shuffle^\eta(T_2)\\
&=\sum_{\eta\in \mu_N} \overline{\eta}^\alpha \sigma\big(  \Phi_\shuffle^\eta(T_1) \big)\Lambda (-\x_\eta)  \x_\eta  \Lambda (\x_\eta)\Phi_\shuffle^\eta(T_2)\\
&=\sum_{\eta\in \mu_N} \overline{\eta}^\alpha \sigma\big( \Phi_\shuffle^\eta (0) \big) e^{T_1\x_\eta} \Lambda(-\x_\eta)  \x_\eta  \Lambda (\x_\eta)e^{-T_2\x_\eta}\Phi_\shuffle^\eta(0)\\
&=\sum_{\eta\in \mu_N} \overline{\eta}^\alpha \sigma\big( \Phi_\shuffle^\eta(0) \big) \frac{\sin(\pi \x_\eta)}{\pi} e^{(T_1-T_2)\x_\eta}\Phi_\shuffle^\eta(0)\\
&=\frac{1}{2\pi i} \sum_{\eta\in \mu_N} \overline{\eta}^\alpha \sigma\big( \Phi_\shuffle^\eta(0) \big) \big(  e^{(\pi i +T_1-T_2)\x_\eta}- e^{(-\pi i +T_1-T_2)\x_\eta}\big) \Phi_\shuffle^\eta(0).
\end{align*}
Letting
\begin{equation*}\label{eq:L_exp} 
\Phi_{exp}(T) = \sigma\big( \Phi_\shuffle(0) \big) e^{T\x_1} \Phi_\shuffle(0)  \in \mathfrak{A}^\vee,\end{equation*}
we have
\begin{equation*} 
\Xi_\alpha  (T_1,T_2)=\frac{1}{2\pi i} \sum_{\eta\in \mu_N} \overline{\eta}^\alpha \big( \Phi_{exp}^\eta (\pi i+T_1-T_2) - \Phi_{exp}^\eta (-\pi i+T_1-T_2)  \big) ,
 \end{equation*}
and so
\begin{equation}\label{eq:xi_int}
\begin{aligned}
\Xi_\alpha  \left( T+\frac{\pi i}{2},T-\frac{\pi i}{2}\right)&=\frac{1}{2\pi i} \sum_{\eta\in \mu_N} \overline{\eta}^\alpha \big( \Phi_{exp}^\eta (2\pi i) - \Phi_{exp}^\eta (0)  \big)\\
&=\frac{1}{2\pi i} \sum_{\eta\in \mu_N} \overline{\eta}^\alpha \big( \Phi_{exp}^\eta (2\pi i) - 1  \big),
\end{aligned}
\end{equation}
where for the last equality we have used $\Phi_{exp}^\eta (0) =\sigma\big( \Phi_\shuffle^\eta(0) \big)\Phi_\shuffle^\eta(0)=1$ (recall Proposition \ref{prop:shuffle_coproduct}).
Since the last term of \eqref{eq:xi_int} does not depend on $T$, the desired result follows from Lemma \ref{lem:xi_alpha_generating_series}.
\end{proof}

\begin{remark}
By definition, the coefficients in $\Phi_{exp}(2\pi i) $ can be written in terms of iterated integrals.
For $a_1,\ldots,a_k\in \{0\}\cup \mu_N$, define
\begin{equation*}\label{eq:int_exp} 
 I_\beta(0';a_1,\ldots,a_k;0')
\end{equation*}
as an iterated integral of $\wedge_{i=1}^k \omega_{a_i}(t_i)$ along the path $\beta$, which is compositions $\beta={\rm dch}\circ \alpha \circ {\rm dch}^{-1}$ of the straight line path ${\rm dch}$ from the tangential basepoints $0'$ to $1'$ and the path $\alpha$ from $1'$ to $1'$ which counterclockwise circle around 1 one times (see also Hirose \cite{H18}).
Using the above integral, we obtain
\[ \langle \Phi_{exp}(2\pi i) ,e_{a_1}\cdots e_{a_k}\rangle =I_\beta(0';a_1,\ldots,a_k;0').\]
From Lemma \ref{lem:xi_alpha_generating_series} and the equation \eqref{eq:xi_int}, the following formula can be proved in much the same as \cite[Corollary 10]{H18}: 
\begin{align*}
 &L^{\mathcal{S}}_\alpha  \binom{\frac{\eta_1}{\eta_0},\ldots,\frac{\eta_{r}}{\eta_{r-1}}}{k_{1},\ldots,k_r}= \frac{1}{2\pi i}\sum_{\eta\in \mu_N} \overline{\eta}^\alpha I_\beta(0';\eta_0\overline{\eta},\{0\}^{k_1-1},\eta_1\overline{\eta},\ldots,\{0\}^{k_r-1},\eta_r\overline{\eta};0'). 
 \end{align*}

\end{remark}

\subsection{Connection with multiple harmonic $q$-sums at roots of unity}

As a result, our symmetric colored multiple zeta values are obtained from an analytic limit of multiple harmonic $q$-sums at primitive roots of unity.

\begin{theorem}\label{thm:connection_smzv}
Let $\alpha \in \Z/N\Z$.
For any index $\binom{\veta}{\kk}\in \mu_N^r\times \Z_{>0}^r$ we have
\[L^{\mathcal{S}}_\alpha \binom{\veta}{\kk}= \lim_{m\rightarrow \infty} \setlength\arraycolsep{1pt}z_{mN+\alpha}  \left(\begin{array}{c}\veta\\\kk\end{array};e^{\frac{2\pi i}{mN+\alpha}} \right) .\]
\end{theorem}
\begin{proof}
This is immediate from Theorems \ref{thm:asym_zm} and \ref{thm:sym_MZV}.
\end{proof}

We give a few remarks on our symmetric colored multiple zeta values. 

\begin{remark}\label{rem:connection_other_models}
\begin{enumerate}
\item Let $\mathcal{Z}$ denote the $\Q(\zeta_N)$-vector space spanned by all colored multiple zeta values of level $N$. 
From Theorem \ref{thm:sym_MZV} and Definition \ref{def:SCMZV} together with the harmonic product formula, we see that our symmetric colored multiple zeta values of level $N$ lie in the space $\mathcal{Z}+ \pi i \mathcal{Z}$.
\item For all $k_1,\ldots,k_r\in \Z_{>0}$ and $\eta_1,\ldots,\eta_r\in \mu_N^r$, our $L^{\mathcal{S}}_1 \binom{\eta_1,\ldots,\eta_r}{k_1,\ldots,k_r}$ coincides with $\zeta^{\rm exp,Ad}(k_r,\ldots,k_1;\eta_r,\ldots,\eta_1,1;0)$ the exponential adjoint cyclotomic multiple zeta value introduced by Jarossay \cite[Eq.~(A.1.3)]{J18}.
\item For any index $\binom{\veta}{\kk}\in \mu_N^r\times \Z_{>0}^r$, it can be shown that
\begin{equation*}\label{eq:SZmodel}
L^{\mathcal{S}}_{-1}\binom{\veta}{\kk}\equiv \zeta_\ast^{\mathcal{S}}\binom{\kk}{\veta} \equiv \zeta_\shuffle^{\mathcal{S}}\binom{\kk}{\veta} \pmod{\pi i \mathcal{Z} + \pi^2 \mathcal{Z}} ,
\end{equation*}
where $\zeta_\bullet^{\mathcal{S}}\binom{\kk}{\veta} \ (\bullet\in\{\ast,\shuffle\})$ are symmetric colored multiple zeta values introduced by Singer and Zhao \cite[Eq.~(3),(4)]{SZ15}.
\end{enumerate}
\end{remark}

Our symmetric colored multiple zeta values originate from the limiting values of multiple harmonic $q$-sums at primitive roots of unity.
This is a completely different perspective from other works on symmetric colored multiple zeta values.

\section{Finite colored multiple zeta values}

\subsection{Definition}

We define the finite colored multiple zeta values as a counterpart of our symmetric colored multiple zeta values for each class $\alpha \in \mathbb{Z}/N\mathbb{Z}$.

Let $\mathcal{P}(N;\alpha)$ be the set of primes congruent to $\alpha$ modulo $N$.
The Chebotarev density theorem shows that the cardinality of the set $\mathcal{P}(N;\alpha)$ is infinite with density $1/\varphi(N)$, where $\varphi$ is Euler's totient function.

If $p$ is prime, since the elements $1-\zeta_p^{j}  \ (1\le j\le p-1)$ are cyclotomic units in $\Z[\zeta_p]$ (see \cite[Proposition 2.8]{W97}), for any index $\binom{\veta}{\kk}\in \mu_N^r\times \Z_{>0}^r$ we have
\begin{equation*}\label{eq:integrality}
 \setlength\arraycolsep{1pt}z_p\left(\begin{array}{c}\veta\\\kk\end{array};\zeta_p\right) \in \Z[\zeta_{pN}].
 \end{equation*}
Let $\mathfrak{P}$ denote a prime ideal in $\mathbb{Z}[\zeta_{pN}]$ above the prime ideal $(1-\zeta_p)$ of $\mathbb{Z}[\zeta_{p}]$ generated by $1-\zeta_p$.
We note that $\zeta_p \equiv 1 \mod \mathfrak{P}$.
For convenience, we think of the residue field $\mathbb{Z}[\zeta_{pN}]/\mathfrak{P}$, which is a finite extension of $\mathbb{F}_p$, as a subfield of the algebraic closure $\overline{\mathbb{F}}_p$.
Under this identification, we easily see that for a prime $p$, $\alpha \in (\mathbb{Z}/N\mathbb{Z})^\times$, $k_1,\ldots,k_r\in \Z_{>0}$ and $\eta_1,\ldots,\eta_r\in \mu_N$ we have
\begin{equation}\label{eq:cong_mod_pe} \setlength\arraycolsep{1pt}z_p \left(\begin{array}{c}\eta_1,\ldots,\eta_r\\ k_1,\ldots,k_r\end{array};\zeta_p \right) \equiv \sum_{p>m_1>\cdots>m_r>0}\frac{\eta_1^{m_1}\cdots \eta_r^{m_r}}{m_1^{k_1}\cdots m_r^{k_r}} \mod \mathfrak{P}. 
\end{equation}

For each $\alpha\in (\Z/N\Z)^\times$, define the ring $\mathcal{A}(\alpha)$ by
\[ \mathcal{A}(\alpha) = \mathcal{A}(N;\alpha):=\prod_{p\in P(N;\alpha)} \overline{\mathbb{F}}_p \big/ \bigoplus_{p\in P(N;\alpha)} \overline{\mathbb{F}}_p\]
Its elements are of the form $(a_p)_p$, where $p$ runs over all primes in $\mathcal{P}(N;\alpha)$ and $a_p\in \overline{\mathbb{F}}_p$. 
Two elements $(a_p)_p$ and $(b_p)_p$ are identified if and only if $a_p=b_p$ for all but finitely many primes $p\in \mathcal{P}(N;\alpha)$.
The rational field $\Q$ can be embedded into $\mathcal{A}(\alpha)$ as follows.
For $a\in \Q$, set $a_p=0$ if $p$ divides the denominator of $a$ and $a_p=a\in \mathbb{F}_p$ otherwise.
Then $(a_p)_p \in \mathcal{A}(\alpha)$ for any $\alpha\in (\Z/N\Z)^\times$.
In this way, we can also embed $\Q(\zeta_N)$ into $\mathcal{A}(\alpha)$.
With this, $\mathcal{A}(\alpha)$ forms a commutative algebra over $\Q(\zeta_N)$.

\

We now define our finite colored multiple zeta values as elements of $\mathcal{A}(\alpha)$.

\begin{definition}\label{def:FCMZV}
Let $\alpha\in (\Z/N\Z)^\times$.
For each index $\binom{\veta}{\kk}=\binom{\eta_1,\ldots,\eta_r}{k_1,\ldots,k_r}\in\mu_N^r\times \Z_{>0}^r$, we define the \emph{finite colored multiple zeta value of level $N$} $L^{\mathcal{A}}_\alpha \binom{\veta}{\kk} $ by
\[ L^{\mathcal{A}}_\alpha \binom{\veta}{\kk} = \left( \sum_{p>m_1>\cdots>m_r>0} \frac{\eta_1^{m_1}\cdots \eta_r^{m_r}}{m_1^{k_1}\cdots m_r^{k_r}} \quad \mod \mathfrak{P}  \right)_{p\in \mathcal{P}(N;\alpha)} \in  \mathcal{A}(\alpha).\]
\end{definition}

Remark that the cases $N=2,3,4,6$ with $\alpha=-1$ were studied by Singer and Zhao \cite[\S3]{SZ15}.
In his study of the Akagi-Hirose-Yasuda type connection with $p$-adic cyclotomic (we call it colored) multiple zeta value, Jarossay \cite[Definition 5.2.2]{J18} introduced another model of finite cyclotomic multiple zeta values as elements of $\prod_p\overline{\mathbb{F}}_p/\bigoplus_p \overline{\mathbb{F}}_p$, where $p$ runs over \emph{all} primes, which will be a different object from ours (see Remark \ref{rem:pp}).

\subsection{Connection with multiple harmonic $q$-sums at roots of unity}

Taking modulo $\mathfrak{P}$, we get $\zeta_p \equiv 1 \mod \mathfrak{P}$. 
This will be an `algebraic' limit $q\rightarrow 1$ mentioned in the introduction.
Collecting multiple harmonic $q$-sums at primitive $p$-th roots of unity modulo $\mathfrak{P}$ for all $p\in \mathcal{P}(N;\alpha)$, we obtain our finite colored multiple zeta values.

\begin{theorem}\label{thm:connection_fmzv}
For $\alpha\in (\Z/N\Z)^\times$ and $\binom{\veta}{\kk}\in\mu_N^r\times \Z_{>0}^r$, we have
\[L^{\mathcal{A}}_\alpha \binom{\veta}{\kk}= \left(\setlength\arraycolsep{1pt}z_p\left(\begin{array}{c}\veta\\\kk\end{array};\zeta_p\right)   \mod \mathfrak{P} \right)_{p\in \mathcal{P}(N;\alpha)}.\]
\end{theorem}
\begin{proof}
This is immediate from \eqref{eq:cong_mod_pe}.
\end{proof}

\section{Fundamental relations}

\subsection{Relations for finite and symmetric colored multiple zeta values}
We will prove the reversal relations and the harmonic relations for finite and symmetric colored multiple zeta values, using those for multiple harmonic $q$-sums at roots of unity.

\begin{proposition} \label{prop:reversal}
 For $\alpha\in \Z/N\Z$, $\bullet \in \{\mathcal{A},\mathcal{S}\}$, positive integers $k_1,\ldots,k_r\ge1$ and $\eta_1,\ldots,\eta_r\in\mu_N$ we have
 \[\overline{L^{\mathcal{\bullet}}_\alpha  \binom{\eta_1,\ldots,\eta_r}{ k_1,\ldots,k_r} } = (-1)^{k_1+\cdots+k_r} \left(\eta_1\cdots \eta_r\right)^{-\alpha} L^{\mathcal{\bullet}}_\alpha \binom{\eta_r,\ldots,\eta_1}{k_r,\ldots,k_1} .\]
\end{proposition}
\begin{proof}
Using the identity $(1-\overline{\zeta}_m)/(1-\zeta_m)=-\overline{\zeta}_m$ and replacing $m_j$ with $m-m_{r+1-j}$, one gets
\begin{align*}
 \overline{\setlength\arraycolsep{1pt}z_m  \left(\begin{array}{c}\eta_1,\ldots,\eta_r \\ k_1,\ldots,k_r\end{array};\zeta_m \right) } &= \sum_{m> m_1>\cdots>m_r>0} \prod_{j=1}^r \overline{\eta}_j^{m_j} \left( \frac{1-\overline{\zeta}_m}{1-\overline{\zeta}_m^{m_j} }\right)^{k_j}  \\
 &= \left( \frac{1-\overline{\zeta}_m}{1-\zeta_m} \right)^{k_1+\cdots+k_r} \sum_{m> m_1>\cdots>m_r>0} \prod_{j=1}^r \eta_j^{-m_j} \left( \frac{1-\zeta _m}{1-\overline{\zeta}_m^{m-m_j} }\right)^{k_j} \\
 &= (-\zeta_m)^{-k_1-\cdots-k_r}\left(\eta_1\cdots \eta_r\right)^{-m}  \sum_{m> m_1>\cdots>m_r>0} \prod_{j=1}^r \eta_j^{m_j} \left( \frac{1-\zeta _m}{1-\zeta_m^{m_j} }\right)^{k_{r+1-j}}\\
 &=\left(-\zeta_m\right)^{-k_1-\cdots-k_r} \left(\eta_1\cdots \eta_r\right)^{-m} \setlength\arraycolsep{1pt}z_m  \left(\begin{array}{c}\eta_r,\ldots,\eta_1 \\ k_r,\ldots,k_1\end{array};\zeta_m \right) .
\end{align*}
With this, the results follow from Theorems \ref{thm:connection_smzv} and \ref{thm:connection_fmzv}.
\end{proof}

We note that $\overline{L^{\mathcal{A}}_\alpha  \binom{\eta_1,\ldots,\eta_r}{ k_1,\ldots,k_r} }$ means  $L^{\mathcal{A}}_\alpha  \binom{\overline{\eta}_1,\ldots,\overline{\eta}_r}{ k_1,\ldots,k_r}$ and that $L^{\mathcal{A}}_\alpha =0$ whenever $\alpha \not\in (\Z/N\Z)^\times$.

\begin{proposition} \label{prop:harmonic}
For $\alpha\in \Z/N\Z$, the $\Q$-linear maps $\mathcal{L}^{\mathcal{S}}_\alpha  : \mathfrak{A}^1_\ast \rightarrow \C$ and $\mathcal{L}^{\mathcal{A}}_\alpha  : \mathfrak{A}^1_\ast \rightarrow \mathcal{A}(\alpha)$ defined by 
\[  \mathcal{L}^{\bullet}_\alpha (e_{k_1,\eta_1}\cdots e_{k_r,\eta_r})= L^{\bullet}_{\alpha} \binom{\eta_1,\ldots,\eta_r}{k_1,\ldots,k_r} \quad \mbox{and}\quad \mathcal{L}^{\bullet}_\alpha (1)=1 \quad (\bullet \in \{\mathcal{S},\mathcal{A}\})\] 
are algebra homomorphisms.
Namely, for any words $w,w'\in \mathfrak{A}^1_\ast$ we have $\mathcal{L}^{\mathcal{\bullet}}_\alpha (w)\mathcal{L}^{\mathcal{\bullet}}_\alpha (w')=\mathcal{L}^{\mathcal{\bullet}}_\alpha (w\ast w')$.
\end{proposition}
\begin{proof}
The result is a consequence of Proposition \ref{prop:stuffle_zm} and Theorems \ref{thm:connection_smzv} and \ref{thm:connection_fmzv}.
\end{proof}

\subsection{Linear shuffle relation}
We show the linear shuffle relation for finite colored multiple zeta values.
Unfortunately, the result is not a consequence of the relation for multiple harmonic $q$-sums at roots of unity and Theorems \ref{thm:connection_smzv} and \ref{thm:connection_fmzv}.

\begin{proposition}\label{prop:shuffle_general}
Let $\alpha\in (\Z/N\Z)^\times$. 
For any $u \in \mathfrak{A}^1$ and $v\in \mathfrak{A}$, we have
\[  \mathcal{L}^{\mathcal{A}}_\alpha  ( {\bf q}(u\shuffle ve_1)) =   (-1)^{|v|+1}\mathcal{L}^{\mathcal{A}}_\alpha  ( {\bf q}(\overset{\leftarrow}{v}e_1 u) ) ,\]
where $|v|$ and $\overset{\leftarrow}{v}$ are respectively the weight and the reversal word.
\end{proposition}
\begin{proof}
We use the same technique as in the proof of Theorem 8.1 in \cite{Kaneko} (see also \cite[Proposition 2.3.3]{J18}).
By abuse of notation, we may view the truncated colored multiple zeta value $L_p$ for a prime $p$ and the iterated integral $I_{[0,z]}$ ($z\in \C$ with $|z|\le 1$) as $\Q$-linear maps $L_p: \mathfrak{A}^1\rightarrow \C$ and $I_{[0,z]}: \mathfrak{A}^1\rightarrow \C$ given by 
\[L_p(e_{k_1,\eta_1}\cdots e_{k_r,\eta_r})=\sum_{p>m_1>\cdots >m_r>0} \frac{\eta_1^{m_1}\cdots \eta_r^{m_r}}{m_1^{k_1}\cdots m_r^{k_r}}\]
and 
\[I_{[0,z]}(e_{k_1,\eta_1}\cdots e_{k_r,\eta_r})= I_{[0,z]}(0^{k_1-1},\eta_1,\ldots,0^{k_r-1},\eta_r).\]
By Proposition \ref{prop:iterated_integral} we have
\begin{align*}
L_p({\bf q}(e_{k_1,\eta_1}\cdots e_{k_r,\eta_r}))= \sum_{0<m<p}\big(\mbox{coefficient of $z^m$ in $I_{[0,z]}(e_{k_1,\eta_1}\cdots e_{k_r,\eta_r})$}\big).
\end{align*}
Using this, for $u=e_{k_1,\eta_1}\cdots e_{k_r,\eta_r}\in \mathfrak{A}^1$ and $w=e_{l_1,\nu_1}\cdots e_{l_s,\nu_s}\in \mathfrak{A}^1$ we compute
\begin{align*}
&L_p({\bf q}(u\shuffle w) )\\
&=\sum_{0<m<p}\big(\mbox{coefficient of $z^m$ in $I_{[0,z]} (u\shuffle w)$}\big)\\
&=\sum_{\substack{0<i,j<p\\i+j<p}} \big(\mbox{coefficient of $z^i$ in $I_{[0,z]} (u)$}\big) \big(\mbox{coefficient of $z^j$ in $I_{[0,z]} ( w)$}\big)\\
&=\sum_{\substack{0<i,j<p\\i+j<p}} \left( \sum_{i>m_2>\cdots>m_r>0} \frac{\eta_1^{i} (\overline{\eta}_1\eta_2)^{m_2} \cdots (\overline{\eta}_{r-1}\eta_r)^{m_r} }{i^{k_1} m_2^{k_2}\cdots m_r^{k_r}}\right)\\
&\times  \left( \sum_{j>n_2>\cdots>n_s>0} \frac{\nu_1^{j} (\overline{\nu}_1\nu_2)^{n_2} \cdots (\overline{\nu}_{s-1}\nu_s)^{n_s} }{j^{l_1} n_2^{l_2}\cdots n_s^{l_s}}\right).
\end{align*}
Since it holds that
\[\sum_{b>n>a}\frac{\eta^n}{n^k} \equiv  \sum_{b>n>a} \frac{\eta^n}{(n-p)^k}\equiv  (-1)^k \sum_{p-a>p-n>p-b} \frac{\eta^n}{(p-n)^k}\mod p  \]
for $0<a,b<p$, the above last term modulo $\mathfrak{P} $ can be reduced to
\begin{align*}
&\equiv \sum_{\substack{0<i,j<p\\i<p-j}} \left( \sum_{i>m_2>\cdots>m_r>0} \frac{\eta_1^{i} (\overline{\eta}_1\eta_2)^{m_2} \cdots (\overline{\eta}_{r-1}\eta_r)^{m_r} }{i^{k_1} m_2^{k_2}\cdots m_r^{k_r}}\right)\\
&\times  (-1)^{l_1+\cdots+l_s} \left( \sum_{p>p-n_s>\cdots>p-n_2>p-j} \frac{\nu_1^{j} (\overline{\nu}_1\nu_2)^{n_2} \cdots (\overline{\nu}_{s-1}\nu_s)^{n_s} }{(p-j)^{l_1} (p-n_2)^{l_2}\cdots (p-n_s)^{l_s}}\right) \\
&\equiv(-1)^{l_1+\cdots+l_s} \nu_s^p \\
&\times\sum_{p>h_s>\cdots>h_2>h>i>m_2>\cdots>m_r>0}\ \frac{(\overline{\nu}_{s-1}\nu_s)^{-h_s}\cdots  (\overline{\nu}_1\nu_2)^{-h_2} \nu_1^{-h}  \eta_1^{i} (\overline{\eta}_1\eta_2)^{m_2} \cdots (\overline{\eta}_{r-1}\eta_r)^{m_r}}{h_s^{l_s}\cdots h_2^{l_2}g^{l_1} i^{k_1}m_2^{k_2}\cdots m_r^{k_r}}\\
&\equiv(-1)^{l_1+\cdots+l_s} \nu_s^p \\
&\times\sum_{p>m_1>\cdots>m_{r+s}>0}\ \frac{(\nu_{s-1}\overline{\nu}_{s})^{m_1}\cdots  (\nu_1\overline{\nu}_2)^{m_{s-1}} \overline{\nu}_1^{m_s}  \eta_1^{m_{s+1}} (\overline{\eta}_1\eta_2)^{m_{s+2}} \cdots (\overline{\eta}_{r-1}\eta_r)^{m_{s+r}}}{m_1^{l_s}\cdots m_{s}^{l_1} m_{s+1}^{k_1}\cdots m_{s+r}^{k_r}}\\
&\equiv(-1)^{l_1+\cdots+l_s} \nu_s^p L_p\binom{\frac{\nu_{s-1}}{\nu_{s}},\ldots,\frac{\nu_1}{\nu_{2}},\frac{1}{\nu_1},\frac{\eta_1}{1},\frac{\eta_2}{\eta_1},\ldots, \frac{\eta_{r}}{\eta_{r-1}}}{l_s,\ldots,l_2,l_1,k_1,k_2,\ldots,k_r} \mod \mathfrak{P} .
\end{align*}
Taking $w=ve_1$ with $v=e_0^{l_1-1}e_{\nu_1}\cdots e_0^{l_{s-1}-1}e_{\nu_{s-1}} e_0^{l_s-1}$, we have
\[ L_p({\bf q}(u\shuffle ve_1) )\equiv (-1)^{l_1+\cdots +l_s} L_p  ( {\bf q}(\overset{\leftarrow}{v}e_1 u) )\mod \mathfrak{P},\]
from which the statement follows.
\end{proof}

\begin{proposition}\label{prop:shuffle_general2}
Let $\alpha\in \Z/N\Z$. 
For any $u \in \mathfrak{A}^1$ and $v\in \mathfrak{A}$, we have
\[  \mathcal{L}^{\mathcal{S}}_\alpha  ( {\bf q}(u\shuffle ve_1)) \equiv  (-1)^{|v|+1}\mathcal{L}^{\mathcal{S}}_\alpha  ( {\bf q}(\overset{\leftarrow}{v}e_1 u) ) \pmod{\pi i\mathcal{Z}},\]
where $\mathcal{Z}$ denotes the $\Q(\zeta_N)$-vector space spanned by all colored multiple zeta value of level $N$.
\end{proposition}
\begin{proof}
It follows from \eqref{eq:xi_int} that
\begin{align*}
\Xi_\alpha \left(T+\frac{\pi i}{2},T-\frac{\pi i}{2}\right) &= \sum_{\eta\in \mu_N} \overline{\eta}^\alpha \left( \sigma(\Phi_\shuffle(0))^\eta \frac{e^{2\pi i \x_\eta}}{2\pi i} \Phi_\shuffle^\eta(0) -\frac{1}{2\pi i}\right)\\
&\equiv  \sum_{\eta\in \mu_N} \overline{\eta}^\alpha\sigma(\Phi_\shuffle^\eta(0)) \x_\eta \Phi_\shuffle^\eta(0) \pmod{\pi i \mathcal{Z}\langle\langle \x_0,\x_\eta \mid\eta\in \mu_N\rangle\rangle}.
\end{align*}
Since $\Phi_\shuffle^\eta(0)$ is group-like, letting $E^\eta=\sigma(\Phi_\shuffle^\eta(0)) \x_\eta \Phi_\shuffle^\eta(0)$, we have
\[ \Delta_\shuffle (E^\eta) =\left( \sigma(\Phi_\shuffle^\eta(0))\otimes \sigma(\Phi_\shuffle^\eta(0)) \right)(\x_\eta\otimes 1+1\otimes \x_\eta) \left( \Phi_\shuffle^\eta(0)\otimes \Phi_\shuffle^\eta(0) \right)= E^\eta \otimes 1+1\otimes E^\eta.\]
For words $w,w'\in \mathfrak{A}$ not being the empty word, by \eqref{eq:delta_shuffle} this shows $\langle E^\eta ,w\shuffle w'\rangle= 0$, and hence, 
\begin{align*}
\left\langle \Xi_\alpha \left(T+\frac{\pi i}{2},T-\frac{\pi i}{2}\right) ,w\shuffle w'\right\rangle&\equiv \left\langle \sum_{\eta\in \mu_N} \overline{\eta}^\alpha E^\eta ,w\shuffle w'\right\rangle\\
&\equiv 0 \pmod{\pi i\mathcal{Z}}.
\end{align*}
For $u,w\in \mathfrak{A}$ with $w=e_{a_1}\cdots e_{a_n}$ it can be shown (see \cite[Eq.~(2.3.3)]{J18} and \cite[Lemma 19]{H18}) that
\[ e_1(u\shuffle w) -(-1)^{|w|}\overset{\leftarrow}{w}e_1 u = \sum_{i=1}^n (-1)^{i+1} e_1 \big(u\shuffle e_{a_1}\cdots e_{a_i}\big) \shuffle e_{a_n}\cdots e_{a_{i+1}}.\] 
Thus, for $u\in \mathfrak{A}^1$ and $w=ve_1$ with $v\in \mathfrak{A}$, using Lemma \ref{lem:xi_alpha_generating_series}, one can compute
\begin{align*}
\mathcal{L}_\alpha^{\mathcal{S}}({\bf q}( u\shuffle ve_1) ) &= \left\langle \Xi_\alpha \left(T+\frac{\pi i}{2},T-\frac{\pi i}{2}\right) ,e_1(u\shuffle ve_1)\right\rangle \\
&\equiv   \left\langle \Xi_\alpha \left(T+\frac{\pi i}{2},T-\frac{\pi i}{2}\right) ,(-1)^{|v|+1}e_1\overset{\leftarrow}{v}e_1 u \right\rangle \\
&\equiv(-1)^{|v|+1} \mathcal{L}_\alpha^{\mathcal{S}}({\bf q}( \overset{\leftarrow}{v}e_1 u ))  \pmod{\pi i\mathcal{Z}}.
\end{align*}
We are done.
\end{proof}

As pointed out by Jarossay, the above proofs for Propositions \ref{prop:shuffle_general} and \ref{prop:shuffle_general2} are the same as the proofs for Lemma 2.3.6 and Proposition 5.2.3 in \cite{J18}.
We call Propositions \ref{prop:shuffle_general} and \ref{prop:shuffle_general2} the linear shuffle relation. 
Singer and Zhao obtains the linear shuffle relation for both finite and symmetric colored multiple zeta values at $\alpha=-1$ (see \cite[Theorems 3.3 and 4.11]{SZ15}), which is a special case of Propositions \ref{prop:shuffle_general} and \ref{prop:shuffle_general2}.

\section{A generalization of the Kaneko-Zagier conjecture}

\subsection{Setup}
We provide some data on finite and symmetric colored multiple zeta values, in order to discuss a generalization of the Kaneko-Zagier conjecture (see \cite[Conjecture 9.5]{Kaneko} for the original statement). 
Hereafter, we denote by $\mathcal{Z}^{\mathcal{A}(N;\alpha)}_k$ (resp. $\mathcal{Z}^{\mathcal{S}(N;\alpha)}_k$) the $\Q(\zeta_N)$-vector space spanned by all finite (resp. symmetric) colored multiple zeta values of weight $k$ and level $N$ with a class $\alpha\in \Z/N\Z$, and set
\[ \mathcal{Z}^{\mathcal{A}(N;\alpha)} = \sum_{k\ge0} \mathcal{Z}^{\mathcal{A}(N;\alpha)}_k,\quad\mathcal{Z}^{\mathcal{S}(N;\alpha)} = \sum_{k\ge0} \mathcal{Z}^{\mathcal{S}(N;\alpha)}_k.\]
By Proposition \ref{prop:harmonic}, these are commutative algebras over the cyclotomic field $\Q(\zeta_N)$.
It is worth mentioning that we do not define $\mathcal{Z}^{\mathcal{A}(N;\alpha)}$ and $\mathcal{Z}^{\mathcal{S}(N;\alpha)}$ as $\Q$-vector spaces, because the reversal relation (Proposition \ref{prop:reversal}) is already a $\Q(\zeta_N)$-linear relation.
Note that since $L_\alpha^{\mathcal{S}}\tbinom{1}{1} = -\pi i$, we always have $2\pi i \in \mathcal{Z}^{\mathcal{S}(N;\alpha)}$.

As we have seen in the previous section, finite and symmetric colored multiple zeta values of level $N$ with a class $\alpha \in \Z/N\Z$ satisfy the same relations (modulo $\pi i$ for symmetric ones), which supports the following conjecture.

\begin{conjecture}\label{conj:gKZ}
For each $\alpha\in (\Z/N\Z)^\times$, the $\Q(\zeta_N)$-linear map
\begin{align*}
\mathcal{Z}^{\mathcal{S}(N;\alpha)}  &\longrightarrow \mathcal{Z}^{\mathcal{A}(N;\alpha)}\\
L^{\mathcal{S}}_\alpha \binom{\veta}{\kk}  &\longmapsto L^{\mathcal{A}}_{\alpha}\binom{\veta}{\kk} 
\end{align*}
is a well-defined algebra homomorphism whose kernel is generated by $2\pi i$.
\end{conjecture}

Conjecture \ref{conj:gKZ} can be viewed as a level $N$ analogue of the Kaneko-Zagier conjecture, which in the case $\alpha=-1$ is proposed in \cite[Conjecture 1.2]{SZ15} for $N=3,4$ and in \cite{Z16} for $N=2$.
There might be a close connection to the $p$-adic variant of the Kaneko-Zagier conjecture for higher levels, proposed by Jarossay \cite[Conjecture 5.3.2]{J18}.
In the following subsections, we give numerical support on Conjecture \ref{conj:gKZ}.

\subsection{Symmetric v.s. classical colored multiple zeta values}

We denote by 
\[ \mathcal{Z}^{(N)} = \sum_{k\ge0} \mathcal{Z}^{(N)}_k\]
the vector space over $\Q$ spanned by all colored multiple zeta value of level $N$.
This is not defined over the cyclotomic field $\Q(\zeta_N)$ because of the result of Deligne-Goncharov below.
The space $\mathcal{Z}^{(N)}$ forms a $\Q$-algebra.
Note that $2\pi i$ lies in $\mathcal{Z}^{(N)}$ of weight 1 if $N\ge3$.
By Definition \ref{def:SCMZV}, for each $\alpha\in \Z/N\Z$ we have
\[ \mathcal{Z}^{\mathcal{S}(N;\alpha)} \subset \mathcal{Z}^{(N)}\otimes_\Q \Q(\zeta_N) \quad (N\ge3),\]
and
\[\mathcal{Z}^{\mathcal{S}(N;\alpha)} \subset \mathcal{Z}^{(N)} +2\pi i \mathcal{Z}^{(N)} \quad (N=1,2).\]
Therefore we have
\[ \dim_{\Q(\zeta_N)} \mathcal{Z}^{\mathcal{S}(N;\alpha)}_k \le \dim_\Q \mathcal{Z}^{(N)}_k \quad (N\ge3).\]

We remark that using Yasuda's result \cite{Y16}, Hirose \cite{H18} proved the equality 
\[\mathcal{Z}^{\mathcal{S}(1;1)}= \mathcal{Z}^{(1)}+2\pi i \mathcal{Z}^{(1)}.\]
The equalities 
\begin{equation}\label{eq:equality_conj} 
\mathcal{Z}^{\mathcal{S}(N;\alpha)} \stackrel{?}{=} \mathcal{Z}^{(N)}\otimes_\Q \Q(\zeta_N) \quad (N\ge3)
\end{equation}
and
\[\mathcal{Z}^{\mathcal{S}(2;\alpha)} \stackrel{?}{=} \mathcal{Z}^{(2)} +2\pi i \mathcal{Z}^{(2)} \]
are open.

\subsection{A work of Deligne-Goncharov}

For comparison, we recall the result of Deligne-Goncharov \cite[\S5]{DG05}.

Let $\overline{\mathcal{Z}}^{(N)} = \mathcal{Z}^{(N)}/2\pi i \mathcal{Z}^{(N)}$ for $N\ge3$ and $\overline{\mathcal{Z}}^{(N)} = \mathcal{Z}^{(N)}/(2\pi i)^2 \mathcal{Z}^{(N)}$ for $N=1,2$.
By constructing motivic fundamental groupoids of $\mathbb{P}^1\backslash\{0,\mu_N,\infty\}$ as an object of the Tannakian category $\mathcal{MT}_N$ of mixed Tate motives over the ring $\mathbb{Z}[\mu_N,\frac{1}{N}]$, Deligne-Goncharov proved that the colored multiple zeta value of level $N$ is a period of $\mathcal{MT}_N$.
As a consequence, we obtain $\dim_\Q \overline{\mathcal{Z}}^{(N)}_k \le \dim_\Q \mathcal{A}^{\mathcal{MT}_N}_k$, where $\mathcal{A}^{\mathcal{MT}_N}=\bigoplus_{k\ge0} \mathcal{A}^{\mathcal{MT}_N}_k$ is the graded Hopf algebra of the pro-unipotent affine group scheme $\mathcal{U}^{\mathcal{MT}_N}$ of the motivic Galois group of $\mathcal{MT}_N$.
It follows from \cite[Theorem 5.24]{DG05} that 
\[ \sum_{k\ge0} \dim_\Q \mathcal{A}^{\mathcal{MT}_N}_k t^k = \begin{cases} \frac{1-t^2}{1-t^2-t^3} & N=1 \\ \frac{1-t^2}{1-t-t^2} & N=2 \\ \frac{1-t}{1-\big(\frac{\varphi(N)}{2}+\nu_N\big) t+ (\nu_N -1) t^2 } & N\ge3 \end{cases},\]
where $\nu_N$ is the number of distinct prime factors of $N$.
Here is a table of $\dim_\Q \mathcal{A}^{\mathcal{MT}_N}_k$:
\begin{center}
\begin{tabular}{c|ccccccccccccccccccccccccc}
$ k$&1&2&3&4&5&6&7&8 \\ \hline
$N=1$ & 0 & 0 & 1 & 0 & 1 & 1 & 1 & 2 \\ \hline
$N=2$ &  1 & 1 & 2 & 3 & 5 & 8 & 13 & 21 \\ \hline
$N=3$ & 1 & 2 & 4 & 8 & 16 & 32 \\ \hline 
$N=4$ & 1 & 2 & 4 & 8 & 16 & 32 \\ \hline
$N=5$ & 2 & 6 & 18 & 54 & 162\\ \hline
$N=6$ & 2 & 5 & 13 & 34 & 89\\ \hline
$N=7$ & 3 & 12 & 48 & 192 & 768 \\ \hline
$N=8$ & 2 & 6 & 18 & 54 & 162\\ \hline
$N=9$ & 3 & 12 & 48 & 192 & 768\\ \hline
$N=10$ & 3 & 11 & 41 & 153 & 571 \\ \hline
\end{tabular}
\end{center}

As further progress on this work, it is proved by Brown \cite{B1} for the case $N=1$ and by Deligne \cite{D10} for the cases $N=2,3,4,8$ that the inequality $\dim_\Q \overline{\mathcal{Z}}^{(N)}_k \le \dim_\Q \mathcal{A}^{\mathcal{MT}_N}_k$ is sharp (see also \cite{Glanois}).
More precisely, in these cases, all periods of $\mathcal{MT}_N$ can be written in terms of colored multiple zeta values of level $N$ (and $2\pi i$ if $N=1,2$).
Unlike these cases, Zhao \cite{Z10} pointed out that there will be periods of $\mathcal{MT}_N$ which can not be written in terms of colored multiple zeta values of level $N$, if $N$ is a prime power with the prime being greater than or equal to 5.

\subsection{Dimension on finite colored multiple zeta values}

We give a table of the conjectural dimension of $\mathcal{Z}^{\mathcal{A}(N;\alpha)}_k$ obtained by a computer and compare it with the result of Deligne-Goncharov \cite{DG05} in the previous subsection. 

Zagier invented an approach to numerically compute the dimension of the $\Q$-vector space spanned by finite multiple zeta value of level 1.
His approach can be applied for the congruence model of finite multiple zeta values of level $N$ with a class $\alpha\in \Z/N\Z$. 
It is defined for positive integers $k_1,\ldots,k_r$ and $f_1,\ldots,f_r\in \Z/N\Z$ by
\[ \zeta^{\mathcal{A}}_\alpha \binom{f_1,\ldots,f_r}{k_1,\ldots,k_r} = \left( \sum_{\substack{p>m_1>\cdots>m_r>0\\m_a\equiv f_a \pmod{N}\ \forall a}} \frac{1}{m_1^{k_1}\cdots m_r^{k_r}} \quad \mod \mathfrak{P}  \right)_{p\in \mathcal{P}(N;\alpha)} \in  \mathcal{A}(\alpha).\]
This model can be written in terms of our finite colored multiple zeta values of level $N$ with a class $\alpha$; 
\begin{equation*}
 \zeta^{\mathcal{A}}_\alpha \binom{f_1,\ldots,f_r}{k_1,\ldots,k_r}  = \frac{1}{N^r} \sum_{\eta_1,\ldots,\eta_r\in \mu_N} \overline{\eta}^{f_1}_1 \cdots \overline{\eta}^{f_r}_r L_\alpha^{\mathcal{A}} \binom{\eta_1,\ldots,\eta_r}{k_1,\ldots,k_r},
\end{equation*}
which is a direct consequence of the well-known identity 
\begin{equation*}\label{eq:ztoZ} 
\frac{1}{N} \sum_{\eta\in \mu_N} \eta^{m} = \begin{cases} 1 & N|m \\ 0 & \mbox{otherwise} \end{cases} \quad (m\in\Z).
\end{equation*}
Moreover, we can prove that the space $\mathcal{Z}^{\mathcal{A}(N;\alpha)}_k$ is generated by all the congruence model of finite multiple zeta values of weight $k$ and level $N$ with a class $\alpha\in \Z/N\Z$ (see \cite{Tasaka?}).

With PARI-GP \cite{PARI}, we numerically counted the number of linearly independent relations over $\Q$ among the congruence model of finite multiple zeta values of weight $k$ and level $N$ with a class $\alpha\in \Z/N\Z$, which may give an upper bound of $\dim_{\Q(\mu_N)} \mathcal{Z}^{\mathcal{A}(N;\alpha)}_k$.
The result tells us that the number of linearly independent relations is seemingly independent from the choices of $\alpha\in (\Z/N\Z)^\times$, namely, we have 
\[\dim_{\Q(\zeta_N)} \mathcal{Z}^{\mathcal{A}(N;\alpha)}_k \stackrel{?}{=} \dim_{\Q(\zeta_N)} \mathcal{Z}^{\mathcal{A}(N;\beta)}_k\]
 for $\alpha,\beta\in (\Z/N\Z)^\times$ with $\alpha\neq \beta$.
Because of this situation, we only display the dimension for the case $\alpha=1$ as follows.

\begin{center}
\begin{tabular}{c|ccccccccccccccccccccccccc}
$ k$&1&2&3&4&5&6&7&8 \\ \hline
$\dim \mathcal{Z}_{k}^{\mathcal{A}(1;1)}$ & 0 & 0 & 1 & 0 & 1 & 1 & 1 & 2 \\ \hline
$\dim \mathcal{Z}_{k}^{\mathcal{A}(2;1)}$ &  1 & 1 & 2 & 3 & 5 & 8 & 13 & 21 \\ \hline
$\dim \mathcal{Z}_{k}^{\mathcal{A}(3;1)}$ & 1 & 2 & 4 & 8 & 16 & 32 \\ \hline 
$\dim \mathcal{Z}_{k}^{\mathcal{A}(4;1)} $ & 1 & 2 & 4 & 8 & 16 & 32 \\ \hline
$\dim \mathcal{Z}_{k}^{\mathcal{A}(5;1)}$ & 2 & 5 & 14 & 39 \\ \hline
$\dim \mathcal{Z}_{k}^{\mathcal{A}(6;1)}$ & 2 & 5 & 13 & 34 \\ \hline
$\dim \mathcal{Z}_{k}^{\mathcal{A}(7;1)}$ & 3 & 10 & 35 \\ \hline
$\dim \mathcal{Z}_{k}^{\mathcal{A}(8;1)}$ & 2 & 6 & 18 & 54\\ \hline
$\dim \mathcal{Z}_{k}^{\mathcal{A}(9;1)}$ & 3 & 12 &48 \\ \hline
$\dim \mathcal{Z}_{k}^{\mathcal{A}(10;1)}$ & 3 &11 & 41  \\ \hline
\end{tabular}
\end{center}

\

This table should be compared with the dimension table of $ \mathcal{A}^{\mathcal{MT}_N}_k$ in the previous subsection.
As a result, we may conjecture the equality
\begin{equation*}\label{eq:conj_finite} 
\dim_\Q \mathcal{A}^{\mathcal{MT}_N} \stackrel{?}{=} \dim_{\Q(\zeta_N)} \mathcal{Z}^{\mathcal{A}(N;\alpha)} \quad (\mbox{for $N=1,2,3,4,6,8,9,10$}),
\end{equation*}
although the above data may not be sufficient.
Assuming Conjectures \ref{conj:gKZ} and \eqref{eq:equality_conj}, the above equality in a certain sense will be true for $N=1,2,3,4,8$ (Brown's and Deligne's cases).
Another perspective is that we may further conjecture that for the cases $N=6,9,10$, all periods of $\mathcal{MT}_N$ may be written in terms of colored multiple zeta values of level $N$.

\begin{remark}\label{rem:pp}
The careful reader will notice that finite colored multiple zeta value may not need to be separated into a class $\alpha \in (\Z/N\Z)^\times$. 
Of course, one can define a variant of the finite colored multiple zeta value as 
\[ \left( \sum_{p>m_1>\cdots>m_r>0} \frac{\eta_1^{m_1}\cdots \eta_r^{m_r}}{m_1^{k_1}\cdots m_r^{k_r}} \quad \mod \mathfrak{P}  \right)_{p}\in \prod_p\overline{\mathbb{F}}_p/\bigoplus_p \overline{\mathbb{F}}_p,\]
where $p$ runs over \emph{all} primes (which is the one introduced by Jarossay \cite[Definition 5.2.2]{J18}).
For this, denote by $\mathcal{Z}^{\mathcal{A}(N)}_k$ the $\Q(\zeta_N)$-vector space spanned by all the above finite multiple zeta values of weight $k$ and level $N$.
We then observed the equality $\dim_{\Q(\zeta_N)} \mathcal{Z}^{\mathcal{A}(N;1)}_k\stackrel{?}{=} \dim_{\Q(\zeta_N)} \mathcal{Z}^{\mathcal{A}(N)}_k$ for $N=3,4$, and the inequality $\dim_{\Q(\zeta_N)} \mathcal{Z}^{\mathcal{A}(N;1)}_k<\dim_{\Q(\zeta_N)} \mathcal{Z}^{\mathcal{A}(N)}_k$ for $N\ge5$.
This observation suggests that the separation with respect to a class will play an important role in the study on finite colored multiple zeta values.
\end{remark}


\end{document}